\documentclass[12pt]{article}
\usepackage{amsfonts}
\usepackage{amssymb}
\usepackage{latexsym}

\usepackage{epsfig}
\usepackage{xypic}

\usepackage[active]{srcltx}

\newenvironment{proof}{\par \noindent{\sc Proof. }}{ \hfill $\square$ \vspace{0.4cm} \par}
\newenvironment{equation*}{$$}{$$ \noindent $\! \!$}

\newtheorem{theorem}{Theorem}[section]
\newtheorem{lemma}[theorem]{Lemma}
\newtheorem{proposition}[theorem]{Proposition}
\newtheorem{corollary}[theorem]{Corollary}
\newtheorem{definition}[theorem]{Definition}

\newcommand {\B}    {{\Bbb B}}
\newcommand {\C}    {{\Bbb C}}
\newcommand {\D}    {{\Bbb D}}
\newcommand {\R}    {{\Bbb R}}

\newcommand {\Z}    {{\Bbb Z}}
\newcommand {\T}    {{\Bbb T}}

\newcommand {\Aa}    {\mathcal{A}}

\newcommand {\Fa}    {\mathcal{F}}

\newcommand {\Oa}    {\mathcal{O}}
\newcommand {\Pa}    {\mathcal{P}}

\newcommand {\Ha}    {\mathcal{H}}
\newcommand {\Va}    {\mathcal{V}}

\newcommand {\esssp}  {\hbox{\rm ess-sp}\,}

\newcommand {\spec}   {\hbox{\rm sp}\,}

\newcommand {\re}     {\hbox{\rm Re}\,}

\newcommand {\II}     {\hbox{\rm II}\,}
\newcommand {\Fol}    {\mathfrak{F}}

\begin{document}

\date{}
\title{Commutative algebras of Toeplitz operators \\ on the Reinhardt domains\thanks{This work was partially supported by CONACYT Projects 46936 and 44620, M\'exico.}}
\author{R. Quiroga-Barranco$^{\dag\, \ddag}$,
and N. Vasilevski$^\dag$ \\
$^\dag$Departamento de Matem\'aticas, CINVESTAV \\
Apartado Postal 14-740, 07000, M\'exico, D.F., M\'exico \\
$^\ddag$Centro de Investigaci\'on en Matem\'aticas \\
Apartado Postal 402, 36000, Guanajuato, Gto., M\'exico\\
e-mail: \ {\sf quiroga@cimat.mx} \\
{\sf nvasilev@math.cinvestav.mx} }

\maketitle

\begin{abstract}
Let $D$ be a bounded logarithmically convex complete Reinhardt
domain in $\mathbb{C}^n$ centered at the origin. Generalizing a
result for the one-dimensional case of the unit disk, we prove that
the $C^*$-algebra generated by Toeplitz operators with bounded
measurable separately radial symbols (i.e., symbols depending only
on $|z_1|$, $|z_2|$, ... , $|z_n|$) is commutative.

We show that  the natural action of the $n$-dimensional torus
$\mathbb{T}^n$ defines (on a certain open full measure subset of
$D$) a foliation which carries a transverse Riemannian structure
having distinguished geometric features. Its leaves are equidistant
with respect to the Bergman metric, and the orthogonal complement to
the tangent bundle of such leaves is integrable to a totally
geodesic foliation. Furthermore, these two foliations are proved to
be Lagrangian.

We specify then the obtained results for the unit ball.
\end{abstract}

\section{Introduction}

A family of recently discovered commutative $C^*$-algebras of
Toeplitz operators on the unit disk (see for details
\cite{Vasil01,Vasil03a}) can be classified as follows. Each pencil
of hyperbolic geodesics determines a set of symbols consisting of
functions which are constant on the corresponding cycles, the
orthogonal trajectories to geodesics forming a pencil. The
$C^*$-algebra generated by Toeplitz operators with such symbols
turns out to be commutative. Moreover, these commutative properties
do not depend at all on smoothness properties of symbols: the
corresponding symbols can be merely measurable. The prime cause
appears to be the geometric configuration of level lines of symbols.
Further it has been proved in \cite{GrudQuirogaVasil05} that,
assuming some natural conditions on the ``richness'' of the symbol
set, the above symbol sets are the only possible which gnerate
commutative Toeplitz operator algebras on each (commonly considered)
weighted Bergman space on the unit disk.

Recall that there are three different types of pencils of hyperbolic
geodesics: an elliptic pencil, which is formed by geodesics
intersecting in a single point, a parabolic pencil, which is formed
by parallel geodesics, and a hyperbolic pencil, which is formed by
disjoint geodesics, i.e., by all geodesics orthogonal to a given
one. Note, that in all cases the cycles are equidistant in the
hyperbolic metric.

The model case for elliptic pencils is when the geodesics intersect
at the origin. In this case the geodesics are diameters and the
cycles are the concentric circles centered at the origin. All other
elliptic pencils can be obtained from this model by means of
M\"obius transformations. The commutative Toeplitz $C^*$-algebra for
the elliptic model case is generated by Toeplitz operators with
radial symbols.

As proved in \cite{GrudKarapVasil03}, the $C^*$-algebras generated
by Toeplitz operators with radial symbols, acting on the weighted
Bergman spaces over the unit ball $\mathbb{B}^n$, are commutative as
well.

In the present paper, we consider a more deep and natural
multidimensional analog of the elliptic model pencil on the unit
disk. We study Toeplitz operators on weighted Bergman spaces over
bounded Reinhardt domains in $\mathbb{C}^n$, and prove, in particular, that the
$C^*$-algebra generated by Toeplitz operators with bounded
measurable separately radial symbols (i.e., symbols depending only on $|z_1|$, $|z_2|$, ... , $|z_n|$) is commutative.

Note that this single result can be also obtained directly by just
calculating the matrix elements $\langle T_a z^p, z^q\rangle$, but
we deliberately follow a more general procedure used in all model
cases on the unit disk (see, for example,  in \cite{Vasil01}). This permits us to construct an analog of the Bargman transform (the operator $R$ restricted on the (weighted) Bergman space), obtain the decomposition of the Bergman projection by means of $R^*$ and $R$, and prepare these operators for the future use.

The second important question treated in the paper is the
understanding of an adequate geometric description which generalizes
geodesics and cycles of the unit disk to a multidimensional case.
Each complete bounded Reinhardt domain $D$ in $\mathbb{C}^n$
centered at the origin admits a natural action of the
$n$-dimensional torus $\mathbb{T}^n$, and this action is isometric
with respect to the Bergman metric in $D$. On a certain open full
measure subset of $D$ this action defines a foliation whose leaves
are all diffeomorphic to $\mathbb{T}^n$. Furthermore, such foliation
carries a transverse Riemannian structure having distinguished
geometric features. First, the leaves are equidistant with respect
to the Bergman metric, and second, the direction perpendicular to
the leaves is totally geodesic, every geodesic which starts in the
perpendicular direction to a leaf stays perpendicular to all other
leaves. Now geometrically: the $C^*$-algebra generated by Toeplitz
operators with bounded measurable symbols, which are constant on the
leaves of the above foliation, is commutative. We prove also that
the orthogonal complement to the tangent bundle of the $\T^n$-orbits
is integrable, thus providing a pair of natural orthogonal
foliations to a Reinhardt domain. Moreover, it turns out that both
foliations are Lagrangian.

It is worth mentioning that the above geometric properties hold for
each pencil of geodesics on the unit disk, but do not hold, for
example, for the case of Toeplitz operators on the unit ball with
radial symbols (the corresponding foliations are not Lagrangian).

We show that the unit ball $\mathbb{B}^n$ in $\mathbb{C}^n$ is the
only Reinhardt domain which is, at the same time, bounded symmetric
and irreducible. Then, we provide a detailed description of the
extrinsic geometry of the foliation by $\T^n$-orbits in the unit
ball. In particular, it is shown, through a computation of its
second fundamental form, that certain geodesics in this foliation
have geodesic curvatures with the same behavior found in the
elliptic pencil of the unit disk. We then consider one-parameter
families of weighted Bergman spaces in $\mathbb{B}^n$, commonly used
in operator theory, and specify the results obtained to this special
case.

Finally, using C. Fefferman's expression for the Bergman kernel of
strictly pseudoconvex domains, we show that for any bounded complete
Reinhardt domain with such pseudoconvexity property, the extrinsic
curvature of the foliation coming from the $\T^n$-action has the
same asymptotic behavior at (suitable) boundary points as the one
observed in the unit ball.

\section{Bergman space on the Reinhardt domains} \label{se:Reinh-Berg}
\setcounter{equation}{0}

Let $D$ be a bounded logarithmically convex complete Reinhardt
domain in $\mathbb{C}^n$ centered at the origin. Denote by $\tau
(D)$ its base, i.e., the set
\begin{equation*}
    \tau(D)= \{r=(r_1,..., r_n)=(|z_1|,...,|z_n|)\, : \
    z=(z_1,...,z_n) \in D\},
\end{equation*}
which belongs to $\R^n_+=\R_+\times ... \times \R_+$.

Consider a positive measurable function (weight) $\mu
(r)=\mu(r_1,...,r_n)$, $r \in \tau(D)$, such that
\begin{equation*}
  \int_{D} \mu (|z|) dv(z)= (2\pi)^n \int_{\tau(D)} \mu(r)
  r dr < \infty,
\end{equation*}
where $dv(z)=dx_1dy_1...dx_ndy_n$ is the usual Lebesgue measure in
$\C^n$, $|z|=(|z_1|,...,|z_n|)$, and $rdr=\prod_{k=1}^n r_kdr_k$. We
assume as well that the weight-function $\mu(r)$ is bounded in some
neighborhood of the origin and does not vanish in this neighborhood.

Introduce the weighted Hilbert space $L_2(D, \mu)$ with the scalar
product
\begin{equation*}
    \langle f , g \rangle = \int_D f(z)\, \overline{g(z)}\, \mu(|z|)
    \, dv(z),
\end{equation*}
and its subspace, the weighted Bergman space $\Aa^2_\mu(D)$, which
consists of all functions analytic in~$D$. We denote as well by
$B_{D,\mu}$ the (orthogonal) Bergman projection of $L_2(D, \mu)$
onto~$\Aa^2_\mu(D)$.

Passing to the polar coordinates $z_k=t_k r_k$, where $t_k \in
\T=S^1$, $k=1,...,n$, and under the identification
\begin{equation*}
    z=(z_1,...,z_n)=(t_1r_1,...,t_nr_n)=(t,r),
\end{equation*}
where $t=(t_1,...,t_n) \in \T^n=\T \times ... \times \T$,
$r=(r_1,..., r_n) \in \tau(D)$, we have $D= \T^n \times \tau(D)$
and
\begin{equation*}
    dv(z)= \prod_{k=1}^n \frac{dt_k}{it_k} \prod_{k=1}^n r_kdr_k.
\end{equation*}
That is we have the following decomposition
\begin{equation*}
    L_2(D,\mu)= L_2(\T^n) \otimes L_2(\tau(D), \mu),
\end{equation*}
where
\begin{equation*}
    L_2(\T^n) = \bigotimes_{k=1}^n L_2(\T,\frac{dt_k}{it_k})
\end{equation*}
and the measure $d\mu$ in $L_2(\tau(D), \mu)$ is given by
\begin{equation*}
    d\mu=\mu(r_1,...,r_n)\,\prod_{k=1}^n r_kdr_k.
\end{equation*}
We note that the Bergman space $\Aa^2_{\mu}(D)$ can be
alternatively defined as the (closed) subspace of $L_2(D, \mu)$
which consists of all functions satisfying the equations
\begin{equation*}
    \frac{\partial}{\partial \overline{z}_k} \varphi =
    \frac{1}{2} \left(\frac{\partial}{\partial x_k} + i \frac{\partial}{\partial
    y_k}\right)\varphi = 0, \ \ \ \ \ k=1,...,n,
\end{equation*}
or, in the polar coordinates,
\begin{equation*}
    \frac{\partial}{\partial \overline{z}_k} \varphi =
    \frac{t_k}{2} \left(\frac{\partial}{\partial r_k} - \frac{t_k}{r_k}
     \frac{\partial}{\partial t_k}\right)\varphi = 0, \ \ \ \ \
     k=1,...,n.
\end{equation*}
Define the discrete Fourier transform ${\Fa}: L_2(\T) \rightarrow
l_2=l_2(\Z)$ by
\begin{equation*}
  {\Fa}: \,f \longmapsto c_n=\frac{1}{\sqrt{2\pi}} \int_{S^1} f(t)\,t^{-n}\,
  \frac{dt}{it}, \ \ \ n \in \Z.
\end{equation*}
The operator $\Fa$ is unitary and
\begin{equation*}
  {\Fa}^{-1}= \Fa^*: \, \{c_n\}_{n \in \Z} \longmapsto f=\frac{1}{\sqrt{2\pi}}
  \sum_{n\in \Z} c_n\,t^n.
\end{equation*}
It is easy to check (see, for example, \cite{Vasil01}, Subsection
4.1) that the operator
\begin{equation*}
    u=(\Fa \otimes I)\frac{t}{2} \left(\frac{\partial}{\partial r} - \frac{t}{r}
     \frac{\partial}{\partial t}\right)(\Fa^{-1} \otimes I) : l_2
     \otimes L_2((0,1), rdr) \longrightarrow  l_2  \otimes L_2((0,1), rdr)
\end{equation*}
acts as follows
\begin{equation*}
    u\, : \ \{c_k(r)\}_{k \in \Z} \longmapsto \left\{ \frac{1}{2}
    \left(\frac{\partial}{\partial r} -
    \frac{k-1}{r}\right)c_{k-1}(r)\right\}_{k \in \Z}.
\end{equation*}
Introduce the unitary operator
\begin{equation*}
    U=\Fa_{(n)} \otimes I \, : \ L_2(\T^n) \otimes L_2(\tau(D), \mu)
    \longrightarrow \ l_2(\Z^n) \otimes L_2(\tau(D), \mu),
\end{equation*}
where $\Fa_{(n)}= \Fa \otimes ... \otimes \Fa$. Then the image
$\Aa_1^2= U(\Aa^2_{\mu}(D))$ of the Bergman space is the closed
subspace of $l_2(\Z^n) \otimes L_2(\tau(D), \mu)$ which consists
of all sequences $\{c_p(r)\}_{p\in \Z^n}$, $r=(r_1,...,r_n) \in
\tau(D)$, satisfying the equations
\begin{equation*}
    \frac{1}{2} \left(\frac{\partial}{\partial r_k} -
    \frac{p_k}{r_k}\right)c_{(p_1,...,p_k)}(r_1,...,r_n) = 0, \ \ \ \ \
    k=1,...,n.
\end{equation*}
These equations are easy to solve, and their general solutions
have the form
\begin{equation*}
    c_p(r)=\alpha_p c_p r^p, \ \ \ \ \ p=(p_1,...,p_n)\in \Z^n,
\end{equation*}
where $c_p \in \C$, $r^p=r_1^{p_1}\cdot ...\cdot r_n^{p_n}$, and
$\alpha_p=\alpha_{|p|}$ (\,$|p|=(|p_1|,...,|p_n|)$,\ in this occurrence) is given by
\begin{equation} \label{eq:alpha_p}
    \alpha_p= \left( \int_{\tau(D)} r^{2|p|}\, \mu(r)\,rdr
    \right)^{-\frac{1}{2}}
    =\left( \int_{\tau(D)} r_1^{2|p_1|}\cdot ... \cdot
    r_n^{2|p_n|}\, \mu(r_1,...,r_n)\,\prod_{k=1}^n r_kdr_k
    \right)^{-\frac{1}{2}}.
\end{equation}
Recall that each function $c_p(r)=\alpha_p c_p r^p$ has to be in
$L_2(\tau(D), \mu)$, which implies that $c_p= 0$ for each
$p=(p_1,...,p_n)$ such that at least one of $p_k < 0$,
$k=1,...,n$.

That is the space $\Aa_1^2 \subset l_2(\Z^n) \otimes L_2(\tau(D),
\mu)$ coincides with the space of all sequences
\begin{equation*}
    c_p(r)= \left\{
    \begin{array}{cl}
      \alpha_p c_p r^p, & p \in \Z_+^n = \Z_+ \times ... \times \Z_+ \\
      0,  & p \in \Z^n \setminus \Z_+^n  \\
    \end{array}
    \right. ,
\end{equation*}
and furthermore
\begin{equation*}
    \|\{c_p(r)\}_{p \in \Z_+^n}\|_{l_2(\Z^n) \otimes L_2(\tau(D),\mu)} =
    \|\{c_p\}_{p \in \Z_+^n}\|_{l_2(\Z^n)}.
\end{equation*}

Introduce now the isometric embedding
\begin{equation*}
    R_0\, : \ l_2(\Z_+^n) \longrightarrow l_2(\Z^n) \otimes L_2(\tau(D),\mu)
\end{equation*}
as follows
\begin{equation*}
    R_0 \, : \ \{c_p\}_{p \in \Z_+^n} \longmapsto
    c_p(r)= \left\{
    \begin{array}{cl}
      \alpha_p c_p r^p, & p \in \Z_+^n  \\
      0,  & p \in \Z^n \setminus \Z_+^n  \\
    \end{array}
    \right. .
\end{equation*}
Then the adjoint operator $R_0^* : l_2(\Z^n) \otimes
L_2(\tau(D),\mu) \longrightarrow l_2(\Z^n)$ is defined by
\begin{equation*}
    R_0^* \, : \ \{f_p(r)\}_{p \in \Z^n} \longmapsto
    \left\{ \alpha_p \int_{\tau(D)} r^p\, f_p(r)\, \mu(r_1,...,r_n)\,
    \prod_{k=1}^n r_kdr_k \right\}_{p \in \Z_+^n},
\end{equation*}
and it is easy to check that
\begin{eqnarray*}
  R_0^* R_0 = I &:& l_2(\Z_+^n) \longrightarrow l_2(\Z_+^n), \\
  R_0 R_0^* =P_1 &:& l_2(\Z^n) \otimes L_2(\tau(D),
                      \mu) \longrightarrow \Aa_1^2,
\end{eqnarray*}
where $P_1$ is the orthogonal projection of $l_2(\Z^n) \otimes
L_2(\tau(D),\mu)$ onto $\Aa_1^2$.

Summarizing the above we have
\begin{theorem}
The operator $R=R_0 U$ maps $L_2(D,\mu)$ onto $l_2(\Z_+^n)$, and
the restriction
\begin{equation*}
    R|_{\Aa^2_{\mu}(D)} \, : \ \Aa^2_{\mu}(D) \longrightarrow l_2(\Z_+^n)
\end{equation*}
is an isometric isomorphism.

The adjoint operator
\begin{equation*}
    R^*= U^*R_0 \, : \ l_2(\Z_+^n) \longrightarrow \Aa^2_{\mu}(D)
    \subset L_2(D,\mu)
\end{equation*}
is the isometrical isomorphism of $l_2(\Z_+^n)$ onto the subspace
$\Aa^2_{\mu}(D)$ of $L_2(D,\mu)$.

Furthermore
\begin{eqnarray*}
  R R^* = I &:& l_2(\Z_+^n) \longrightarrow l_2(\Z_+^n), \\
  R^* R =B_{D,\mu} &:& L_2(D,\mu) \longrightarrow \Aa^2_{\mu}(D),
\end{eqnarray*}
where $B_{D,\mu}$ is the Bergman projection of $L_2(D,\mu)$ onto
$\Aa^2_{\mu}(D)$.
\end{theorem}

\begin{theorem}
The isometric isomorphism
\begin{equation*}
  R^*=U^*R_0 : l_2(\Z_+^n) \longrightarrow {\Aa}^2_{\mu}(D)
\end{equation*}
is given by
\begin{equation} \label{eq:unitR*}
  R^* : \, \{c_p\}_{p \in \Z_+^n} \ \longmapsto \
  (2\pi)^{-\frac{n}{2}}\sum_{p \in \Z_+^n}\alpha_p \,c_p\, z^p.
\end{equation}
\end{theorem}

\begin{proof}
Calculate
\begin{eqnarray*}
R^*=U^*R_0 &:& \{c_p\}_{p \in \Z_+^n} \longmapsto U^*
(\{\alpha_p \,c_p\, r^p\}_{p \in \Z_+^n} ) \\
&=&(2\pi)^{-\frac{n}{2}}\sum_{p \in \Z_+^n}\alpha_p \,c_p\, (tr)^p
= (2\pi)^{-\frac{n}{2}}\sum_{p \in \Z_+^n}\alpha_p \,c_p\, z^p.
\end{eqnarray*}
\end{proof}

\begin{corollary}
The inverse isomorphism
\begin{equation*}
  R : {\Aa}_{\mu}^2(D) \longrightarrow l_2(\Z_+^n)
\end{equation*}
is given by
\begin{equation} \label{eq:unitR}
   R : \varphi (z) \ \longmapsto \ \left\{(2\pi)^{-\frac{n}{2}} \,
  \alpha_p\, \int_{D} \varphi(z)\, \overline{z}^p\,
  \mu(|z|) \,dv(z)\right\}_{p \in \Z_+^n}.
\end{equation}
\end{corollary}

\section{Toeplitz operators with separately radial symbols} \label{se:Reinh-Toepl}
\setcounter{equation}{0}

We will call a function $a(z)$, $z \in D$, {\em separately radial} if
$a(z)=a(r)=a(r_1,...,r_n)$, i.e., $a$ depends only on the radial
components of $z=(z_1,...,z_n)=(t_1r_1,...,t_nr_n)$.

\begin{theorem} \label{th:RT_aR*}
Let $a=a(r)$ be a bounded measurable separately radial function. Then
the Toeplitz operator $T_a$ acting on $\Aa^2_{\mu}(D)$ is unitary
equivalent to the multiplication operator $\gamma_aI=R\,T_a R^*$
acting on $l_2(\Z_+^n)$, where $R$ and $R^*$ are given by
(\ref{eq:unitR}) and (\ref{eq:unitR*}) respectively. The sequence
$\gamma_a=\{\gamma_a(p)\}_{p \in \Z_+^n}$ is given by
\begin{equation} \label{eq:gamma}
 \gamma_a(p)=\alpha_p^2 \int_{\tau(D)} a(r)\,r^{2p}\,
   \mu(r_1,...,r_n)\, \prod_{k=1}^n r_kdr_k, \ \ \ \ \ \ p \in \Z_+^n.
\end{equation}
\end{theorem}

\begin{proof}
The operator $T_a$ is obviously unitary equivalent to the operator
\begin{eqnarray*}
  R\,T_a\,R^* &=& R\,B_{D,\mu} a B_{D,\mu}R^*=R(R^*R)a(R^*R)R^* \\
  &=&(RR^*)RaR^*(RR^*)=RaR^* \\
  &=& R_0^*Ua(r)U^{-1}R_0 \\
  &=&R_0^*(\Fa_{(n)}\otimes I)a(r)(\Fa_{(n)}^{-1}\otimes I)R_0 \\
  &=&R_0^*a(r)R_0.
\end{eqnarray*}
Now
\begin{eqnarray*}
  R_0^*a(r)R_0\{c_p\}_{p \in \Z_+^n} &=&
  R_0^* \left\{ a(r)\, \alpha_p\, c_p\, r^p \right\}_{p \in \Z_+^n}
  \\
  &=& \left\{ \alpha_p \int_{\tau(D)} r^p\, a(r)\, \alpha_p\, c_p\, r^p\,
  \mu(r_1,...,r_n)\, \prod_{k=1}^n r_kdr_k \right\}_{p \in \Z_+^n} \\
  &=& \{\gamma_a(p)\cdot c_p\}_{p \in \Z_+^n},
\end{eqnarray*}
where
\begin{equation*}
  \gamma_a(p)=\alpha_p^2 \int_{\tau(D)} a(r)\,r^{2p}\,
   \mu(r_1,...,r_n)\, \prod_{k=1}^n r_kdr_k, \ \ \ \ \ \ p \in \Z_+^n.
\end{equation*}
\end{proof}

It is easy to see that the system of functions $\{e_p\}_{p \in
\Z_+^n}$, where $e_p(z)=(2\pi)^{-\frac{n}{2}}\,\alpha_p\, z^{p}$,
forms an orthonormal base in $\Aa^2_{\mu}(D)$.

\begin{corollary}
The Toeplitz operator $T_a$ with bounded measurable separately radial
symbol $a(r)$ is diagonal with respect to the above orthonormal
base:
\begin{equation} \label{eq:diag}
    T_a \, e_p = \gamma_a(p) \cdot  \, e_p,
    \ \ \ \ \ \ p \in \Z_+^n.
\end{equation}
\end{corollary}

We can easily extend the notion of the Toeplitz operator for
measurable {\em unbounded} separately radial symbols. Indeed, given a
symbol $a=a(r) \in L_1(\tau(D), \mu)$, we still have equality
(\ref{eq:diag}). Then the densely defined (on the finite linear
combinations of the above base elements) Toeplitz operator can be
extended to a bounded operator on a whole $\Aa^2_{\mu}(D)$ if and
only if the sequence $\gamma_a=\{\gamma_a(p)\}_{p \in \Z_+^n}$ is
bounded. That is we have

\begin{corollary} \label{co:bnd&com}
The Toeplitz operator $T_a$ with separately radial symbol $a=a(r) \in
L_1(\tau(D), \mu)$ is bounded on $\Aa^2_{\mu}(\D)$ if and only if
\begin{equation*}
  \gamma_a=\{\gamma_a(p)\}_{p \in \Z_+^n} \in l_{\infty},
\end{equation*}
and
\begin{equation*}
  \|T_a\|=\sup_{p \in \Z_+^n} |\gamma_a(p)|.
\end{equation*}
The Toeplitz operator $T_a$ is compact if and only if $\gamma
_{a}\in c_{0}$ that is
\begin{equation*}
  \lim_{p\rightarrow \infty} \gamma_a(p) = 0.
\end{equation*}
The spectrum of the bounded Toeplitz operator $T_a$ is given by
\begin{equation*}
  \spec T_a = \overline{\{\gamma_a(p): \, p \in \Z_+^n\}},
\end{equation*}
and its essential spectrum $\esssp T_a$ coincides with the set of
all limit points of the sequence~$\{\gamma_a(p)\}_{p \in \Z_+^n}$.
\end{corollary}

\begin{corollary}
The $C^*$-algebra generated by Toeplitz operators with
separately radial $L_{\infty}$-symbols is commutative.
\end{corollary}

\section{Foliations, transverse Riemannian structures, and \\ bundle-like metrics}

In this section we will briefly summarize some notions of foliations
and their geometry. We refer to \cite{Molino88} for further details.

A foliation on a manifold $M$ is a partition of $M$ into connected
submanifolds of the same dimension that locally looks like a
partition given by the fibers of a submersion. The local picture is
given by considering foliated charts and the partition as a global
object is obtained by imposing a compatibility condition between the
foliated charts. We make more precise this notion through the
following definitions.

\begin{definition}
On a smooth manifold $M$ a codimension $q$ foliated chart is a pair
$(\varphi, U)$ given by an open subset $U$ of $M$ and a smooth
submersion $\varphi : U \rightarrow V$, where $V$ is an open subset
of $\mathbb{R}^q$. For a foliated chart $(\varphi, U)$ the connected
components of the fibers of $\varphi$ are called the plaques of the
foliated chart. Two codimension $q$ foliated charts
$(\varphi_1,U_1)$ and $(\varphi_2,U_2)$ are called compatible if
there exists a diffeomorphism $\psi_{12} : \varphi_1(U_1\cap U_2)
\rightarrow \varphi_2(U_1\cap U_2)$ such that the following diagram
commutes:
\begin{equation}\label{charts-comp}
\xymatrix{
        &   U_1\cap U_2 \ar[dl]_{\varphi_1} \ar[dr]^{\varphi_2}  & \\
   \varphi_1(U_1\cap U_2) \ar[rr]^{\psi_{12}}    &     &   \varphi_2(U_1\cap U_2)
}
\end{equation}
A foliated atlas on a manifold $M$ is a collection
$\{(\varphi_\alpha,U_\alpha)\}_\alpha$ of foliated charts which are
mutually compatible and that satisfy $M = \bigcup_\alpha U_\alpha$.
\end{definition}

It is straightforward to check that the compatibility of two
foliated charts $(\varphi_1,U_1)$ and $(\varphi_2,U_2)$ ensures
that, when restricted to $U_1\cap U_2$, both submersions $\varphi_1$
and $\varphi_2$ have the same plaques. This in turn implies that,
for any given foliated atlas, the following is an equivalence
relation in $M$.
\begin{eqnarray*}
        x\sim y &\iff& \mbox{there is a sequence of plaques }
                                (P_k)_{k=0}^l \mbox{ for foliated charts }
                                (\varphi_k, U_k)_{k=0}^l       \\
                        &&      \mbox{of the foliated atlas, such that }
                                x\in P_0,\ y\in P_l,   \\
                        &&      \mbox{and }
                                P_{k-1}\cap P_k\neq\phi \mbox{ for every }
                                k=1,\dots,l
\end{eqnarray*}
We will refer to the latter as the equivalence relation of the
foliated atlas. It is a simple matter to show that the equivalence
classes are in fact submanifolds of $M$ of dimension $\dim(M)-q$,
where $q$ is the (common) codimension of the foliated charts.

\begin{definition}\label{def-foliation}
A foliation $\Fol$ on a manifold $M$ is a partition of $M$ which can
be described as the classes of the equivalence relation of a
foliated atlas. The classes are called the leaves of the foliation.
\end{definition}

Suppose that $M$ is a manifold carrying a smooth foliation $\Fol$.
We will denote with $T\Fol$ the vector subbundle of $TM$ that
consists of elements tangent to the leaves of $\Fol$. We can
consider the quotient bundle $TM/T\Fol$ which we will denote by
$T^t\Fol$. The latter will be referred to as the transverse vector
bundle of the foliation $\Fol$. Since $T^t\Fol$ is a smooth vector
bundle, we can consider the associated linear frame bundle which we
will denote with $L_T(\Fol)$. More precisely, we have as a set:
$$
        L_T(\Fol) = \{ A : \ A : \mathbb{R}^q \rightarrow
        T^t_x\Fol = T_xM/T_x\Fol \mbox{ is an isomorphism and }
        x\in M \},
$$
where $q$ is the codimension of $\Fol$ in $M$. It is easily seen
that $L_T(\Fol)$ is a principal fiber bundle with structure group
$\mathrm{GL}_q(\R)$, we refer to \cite{Molino88} for the
details of the proof. The principal bundle $L_T(\Fol)$ is called the
transverse frame bundle since it allows us to study the geometry
transverse to the foliation $\Fol$.

When studying the transverse geometry of a foliation $\Fol$ it is
useful to consider a certain natural foliation in $L_T(\Fol)$, which
is defined as follows.


Suppose that for a foliation $\Fol$ on a manifold $M$ we choose a
foliated atlas $\{(\varphi_\alpha,U_\alpha)\}_\alpha$ that
determines the foliation as in Definition \ref{def-foliation}. For
any foliated chart $(\varphi_\alpha,U_\alpha)$ and every $x\in
U_\alpha$ we have a linear map $d(\varphi_\alpha)_x : T_xM
\rightarrow \mathbb{R}^q$ whose kernel is $T_x\Fol$. This induces a
linear isomorphism $d(\varphi_\alpha)^t_x : T^t_x\Fol = T_xM/T_x\Fol
\rightarrow \mathbb{R}^q$. The latter allows us to define the smooth
map:
\begin{eqnarray*}
    \varphi^{(1)}_\alpha : L_T(\Fol|_{U_\alpha}) &\rightarrow&
L(V_\alpha) \\
    A &\mapsto& d(\varphi_\alpha)^t_x \circ A,
\end{eqnarray*}
where $L_T(\Fol|_{U_\alpha})$ is the open subset of $L_T(\Fol)$
given by inverse image of $U_\alpha$ under the natural projection
$L_T(\Fol) \rightarrow M$, $A$ is mapped to $x$ under such
projection and $V_\alpha$ is the target of $\varphi_\alpha$. Next we
observe that, since $V_\alpha$ is open in $\mathbb{R}^q$, the
manifold $L(V_\alpha)$ is open in $\mathbb{R}^q\times
\mathrm{GL}_q(\R)$ and so it is open in $\mathbb{R}^{q+q^2}$ as
well. Furthermore, from our choices it is easy to check that the
commutative diagram (\ref{charts-comp}) and the compatibility of
charts in a foliated atlas induce a corresponding commutative
diagram given by:
$$
\xymatrix{
        &   L_T(\Fol|_{U_1\cap U_2}) \ar[dl]_{\varphi_{\alpha_1}^{(1)}}
    \ar[dr]^{\varphi_{\alpha_2}^{(1)}}  & \\
   L(\varphi_{\alpha_1}(U_1\cap U_2)) \ar[rr]^{\psi_{\alpha_1\alpha_2}^{(1)}}    &     &
L(\varphi_{\alpha_2}(U_1\cap U_2)) }
$$
where $\psi_{\alpha_1\alpha_2}^{(1)}$ is defined as above for the
diffeomorphism $\psi_{\alpha_1\alpha_2}$ for which we have
$\varphi_{\alpha_2} = \psi_{\alpha_1\alpha_2} \circ
\varphi_{\alpha_1}$, as in diagram (\ref{charts-comp}). This shows
that the set
$\{(\varphi_\alpha^{(1)},L_T(\Fol|_{U_\alpha}))\}_\alpha$ defines a
foliated atlas. The corresponding foliation in $L_T(\Fol)$ is called
the lifted foliation. We state without proof the following result
which can be found in \cite{Molino88}.

\begin{theorem}\label{lifted-fol-thm}
Let $\Fol$ be a foliation on a smooth manifold $M$. Then, the
natural projection $L_T(\Fol) \rightarrow M$ maps the leaves of the
lifted foliation of $L_T(\Fol)$ locally diffeomorphically onto the
leaves of $\Fol$.
\end{theorem}

From its construction, the principal fiber bundle $L_T(\Fol)
\rightarrow M$ models some aspects of the transverse geometry of the
foliation $\Fol$. At the same time, Theorem \ref{lifted-fol-thm}
shows that the lifted foliation in $L_T(\Fol)$ is needed to fully
capture the foliated nature of the transverse geometry of $\Fol$. In
order to define transverse geometric structures for a given
foliation $\Fol$ we consider now reductions of $L_T(\Fol)$
compatible with the lifted foliation. More precisely, we have the
following definition which also introduces the notion of a
Riemannian foliation.

\begin{definition}\label{def-Riem-fol}
Let $M$ be a manifold carrying a smooth foliation $\Fol$ of
codimension $q$, and let $H$ be a Lie subgroup of
$\mathrm{GL}_q(\R)$. A transverse geometric $H$-structure is a
reduction $Q$ of $L_T(\Fol)$ to the subgroup $H$ which is saturated
with respect to the lifted foliation, i.e. such that $Q\cap L \neq
\phi$ implies $L \subset Q$ for every leaf $L$ of the lifted
foliation. A transverse geometric $O(q)$-structure is also called a
transverse Riemannian structure. A foliation endowed with a
transverse Riemannian structure is called a Riemannian foliation.
\end{definition}

From the definition, it is easy to see that a transverse Riemannian
structure defines a Riemannian metric on the bundle $TM/T\Fol =
T^t\Fol$. However, a transverse Riemannian structure is more than a
simple Riemannian metric on $T^t\Fol$. By requiring the
$O(q)$-reduction that defines a transverse Riemannian structure to
be saturated with respect to the lifted foliation, as in Definition
\ref{def-Riem-fol}, we ensure the invariance of the metric as we
move along the leaves in $M$. This is a well known property of
Riemannian foliations whose further discussion can be found in
\cite{Molino88} and other books on the subject. Here we observe
that, since a Riemannian metric on a manifold defines a distance,
the invariance of a transverse Riemannian structure as we move along
the leaves can be interpreted as the leaves of the foliation in $M$
to be equidistant while we move along them. Again, this sort of
remark is well known in the theory of foliations and shows that a
Riemannian foliation has a distinguished geometry. In particular,
not every foliation admits a Riemannian structure, a standard
example is given by the Reeb foliation of the sphere $S^3$ (see
\cite{Molino88}).

A fundamental way to construct transverse Riemannian structures for
a foliation is to consider suitable Riemannian metrics on the
manifold that carries the foliation. To describe such construction
we will need some additional notions.

\begin{definition}
Let $\Fol$ be a smooth foliation on a manifold $M$. A vector field
$X$ on $M$ is called foliate if for every vector field $Y$ tangent
to the leaves of $\Fol$ the vector field $[X,Y]$ is tangent to the
leaves as well.
\end{definition}

From the previous definition, we observe that the set of foliate
vector fields is the normalizer of the fields tangent to the leaves
of $\Fol$ in the Lie algebra of all vector fields on $M$.

\begin{definition}
Let $\Fol$ be a smooth foliation on a manifold $M$. A Riemannian
metric $h$ in $M$ is called bundle-like for the foliation $\Fol$ if
the real-valued function $h(X,Y)$ is constant along the leaves of
$\Fol$ for every pair of vector fields $X$, $Y$ which are foliate
and perpendicular to $T\Fol$ with respect to $h$.
\end{definition}

Suppose that $h$ is a Riemannian metric on a manifold $M$ and that
$\Fol$ is a foliation on $M$. Then, the canonical projection $TM
\rightarrow T^t\Fol$ allows us to induce a Riemannian metric on the
bundle $T^t\Fol$, which in turn provides an $O(q)$-reduction of the
transverse frame bundle $L_T(\Fol)$ (where $q$ is the codimension of
$\Fol$). Nevertheless, such reduction does not necessarily defines a
transverse Riemannian structure. The next result states that
bundle-like metrics are precisely those that define transverse
Riemannian structures. The proof of this theorem can be found in
\cite{Molino88}.

\begin{theorem}\label{bundle-like-vs-Riem-fol}
Let $M$ be a manifold carrying a smooth foliation $\Fol$ of
codimension $q$. For every Riemannian metric $h$ on $M$, denote by
$O_T(M,h)$ the $O(q)$-reduction of $L_T(\Fol)$ given by the
Riemannian metric on $T^t(\Fol)$ coming from $h$ and the natural
projection $TM \rightarrow T^t\Fol$. If $h$ is a bundle-like metric,
then $O_T(M,h)$ defines a transverse Riemannian structure on $\Fol$.
Conversely, for every transverse Riemannian structure given by a
reduction $Q$ as in Definition \ref{def-Riem-fol}, there is a
bundle-like metric $h$ on $M$ such that $Q = O_T(M,h)$.
\end{theorem}

Based on this result, we give the following definition.

\begin{definition}\label{def-bundle-like-compatible}
Let $\Fol$ be a Riemannian foliation on a manifold $M$. We will say
that a bundle-like metric $h$ on $M$ is compatible with the
Riemannian foliation if $O_T(M,h)$ is the reduction which defines
the corresponding transverse Riemannian structure.
\end{definition}

A fundamental property of Riemannian foliations is that, with
respect to compatible bundle-like metrics, geodesics which start
perpendicular to a leaf of the foliation stay perpendicular to all
leaves. 

\begin{theorem}\label{transverse-tot-geod}
Let $\Fol$ be a Riemannian foliation on a manifold $M$ and let $h$
be a compatible bundle-like metric. If $\gamma$ is a geodesic of $h$
such that $\gamma'(t_0) \in (T_{\gamma(t_0)}\Fol)^\perp$, for some
$t_0$, then $\gamma'(t) \in (T_{\gamma(t)}\Fol)^\perp$ for every
$t$.
\end{theorem}

This theorem is fundamental in the theory of Riemannian foliations
and its proof can be found in \cite{Molino88}. We can provide its
geometric interpretation as follows. Let $M$, $\Fol$ and $h$ be as
in Theorem \ref{transverse-tot-geod}, and denote with $T\Fol^\perp$
the orthogonal complement of $T\Fol$ in $TM$; in particular, $TM =
T\Fol \oplus T\Fol^\perp$. Hence, Theorem \ref{transverse-tot-geod}
states that every geodesic with an initial velocity vector in
$T\Fol^\perp$ has velocity vector contained in $T\Fol^\perp$ for all
time.

In a sense, the above states that the orthogonal complement
$T\Fol^\perp$ contains all geodesics perpendicular to $T\Fol$. If
the codimension of $\Fol$ is $1$, then $T\Fol^\perp$ is
one-dimensional and it can be integrated to a smooth one-dimensional
foliation $\Fol^\perp$ whose leaves are perpendicular to those of
$\Fol$. In such case, Theorem \ref{transverse-tot-geod} ensures that
the leaves of $\Fol^\perp$ are geodesics with respect to the
bundle-like metric $h$.

If $\Fol$ has codimension greater than $1$, then we can still
consider the possibility of $T\Fol^\perp$ to be integrable, e.g. to
satisfy the hypothesis of Frobenius theorem (see \cite{Warner83}).
In such case, we do have a foliation $\Fol^\perp$ whose leaves are
orthogonal to those of $\Fol$. Again, in this case, Theorem
\ref{transverse-tot-geod} implies that the leaves of $\Fol^\perp$
are totally geodesic. At the same time, the vector bundle
$T\Fol^\perp$ is not always integrable.
Nevertheless, the above discussion shows that $T\Fol^\perp$ can be
thought of as being totally geodesic from a broader viewpoint.
Alternatively, we can say that, from a geometric point of view, the
foliation $\Fol$ is transversely totally geodesic.

It is worth mentioning that the integrability of the bundle
$T\Fol^\perp$ given by a Riemannian foliation and a bundle-like
metric is not at all trivial and requires strong restrictions on the
geometry of the foliation or its leaves. As an example, we refer to
\cite{Quiroga-Ann}, where the integrability of the corresponding
$T\Fol^\perp$ is only obtained for leaves carrying a suitable
nonpositively curved Riemannian metric. At the same time, we will
prove in the following sections that the orthogonal complement to
the tangent bundle of the $\T^n$-orbits in a Reinhardt domain is
integrable, which will then imply the presence of strong geometric
features on such domains.

\section{Extrinsic geometry of foliations}\label{se:Extrinsic-fol}

For a submanifold of any Riemannian manifold one can measure the
obstruction for the submanifold to be a totally geodesic in the
ambient. This also measures the extrinsic curvature of the
submanifold, which is determined by the particular embedding and not
just the inherited metric. We now briefly discuss some well known
methods to study this extrinsic curvature and refer to
\cite{KobNomizu96-2} and \cite{Oneill83} for further details. For
our purposes it will be convenient and natural to discuss these
notions for foliations.

Let $\widehat{M}$ be a Riemannian manifold and $\Fol$ be a foliation
of $\widehat{M}$ having codimension $q$ and with $p$-dimensional
leaves. We will denote by $\widehat{\nabla}$ the Levi-Civita
connection of $\widehat{M}$ and by $\nabla$ the connection of the
bundle $T\Fol$ obtained by pasting together the Levi-Civita
connections of the leaves of $\Fol$ for the metric inherited from
that of $\widehat{M}$. Let us also denote by $\Va$ and $\Ha$ the
orthogonal projections of $T\widehat{M}$ onto $T\Fol$ and
$T\Fol^\perp$, respectively. These projections are respectively
called the vertical and horizontal projections with respect to
$\Fol$. Then the following holds (see \cite{KobNomizu96-2}):

\begin{lemma}
The connection $\nabla$ is the vertical projection of
$\widehat{\nabla}$. More precisely, we have:
$$
    \nabla_X Y = \Va(\widehat{\nabla}_X Y),
$$
for every pair of vector fields in $\widehat{M}$ everywhere tangent
to the leaves of $\Fol$.
\end{lemma}

We recall that the Levi-Civita connection is the differential
operator that allows to define geodesics. Hence, the previous result
shows that the obstruction for the leaves of $\Fol$ to be totally
geodesic is precisely the difference between $\widehat{\nabla}$ and
its vertical projection as above, in other words, the horizontal
projection of $\widehat{\nabla}$. This suggests to introduce the
following classical definition (see \cite{KobNomizu96-2}).

\begin{definition}
Let $\Fol$ be a foliation of a Riemannian manifold $\widehat{M}$.
The second fundamental form $\II$ of the leaves of $\Fol$ is given
at every $x \in \widehat{M}$ by:
\begin{eqnarray*}
  \II_x : T_x\Fol\times T_x\Fol &\rightarrow& T_x\Fol^\perp \\
  (u,v) &\mapsto& \Ha(\widehat{\nabla}_X Y)_x,
\end{eqnarray*}
where $X,Y$ are vector fields defined in a neighborhood of $x$ in
$\widehat{M}$ everywhere tangent to $\Fol$ and such that $X_x = v$
and $Y_x = v$.
\end{definition}

It is very well known that the definition of $\II_x$ as above does
not depend on the choice of the vector fields $X$ and $Y$. It is
also known that the second fundamental form at every every point is
a symmetric bilinear form that defines a tensor which is a section
of the bundle $T\Fol^*\otimes T\Fol^* \otimes T\Fol^\perp$.

As it occurs with any tensor, it is easier to describe some of the
properties of $\II$ by introducing local bases for the bundles
involved and computing the components with respect to such bases.
This is particularly useful if one has global bases for the bundles.
These are defined more precisely as follows.

\begin{definition}
Let $E \rightarrow \widehat{M}$ be a subbundle of the tangent bundle
of the Riemannian manifold $\widehat{M}$. Then, a collection
$(V_1,\dots,V_k)$ of sections of $E$ defined on all of $\widehat{M}$
is called a global framing of $E$ if for every $x\in \widehat{M}$,
the set of tangent vectors $(V_1(x),\dots,V_k(x))$ is a basis for
the fiber $E_x$.
\end{definition}

The next result is an obvious consequence of the symmetry of $\II$.
It will allow us to simplify the computation of the values for
$\II$.

\begin{proposition}\label{prop:II-quadratic}
Let $\widehat{M}$ be a Riemannian manifold with a foliation $\Fol$
as above. Suppose that $(V_k)_{k=1}^p$ is a global framing of
$T\Fol$. Then $\II$ as a tensor is completely determined by the
vector fields $\II(V_k + V_l,V_k + V_l)$ for $k,l=1,\dots,p$.
\end{proposition}
\begin{proof}
It is enough to use the relation:
$$
    \II(V_k,V_l) = \frac{1}{2}(\II(V_k + V_l,V_k + V_l) - \II(V_k,V_k) - \II(V_l,V_l))
$$
which is satisfied by the symmetry of $\II$.
\end{proof}

\section{Isometric actions of Lie groups}

In this section we will consider some general notions about actions
of Lie groups on a manifold preserving a Riemannian metric. In what
follows $M$ will denote a smooth manifold and $G$ a connected Lie
group acting smoothly on the left on $M$. For the next definition,
we recall that the stabilizer of a point $x\in M$ for the $G$-action
is the set $G_x = \{ g\in G : gx = x\}$.

\begin{definition}
The action of $G$ on $M$ is called free (locally free) if for every
$x \in M$ the stabilizer $G_x$ is trivial (respectively discrete).
\end{definition}

A straightforward application of Frobenius theorem on the
integrability of vector subbundles of a tangent bundle (see
\cite{Warner83}) allows us to obtain the following result.

\begin{proposition}\label{G-orbits-fol}
If $G$ acts locally freely on $M$, then the $G$-orbits define a
smooth foliation on $M$.
\end{proposition}
\begin{proof}
Denote by $\mathfrak{g}$ the Lie algebra of $G$. Then for every
$X\in \mathfrak{g}$ we can define the transformations of $M$ given
by the maps:
\begin{eqnarray*}
    \varphi_t: M &\rightarrow& M \\
    x &\mapsto& \exp(tX) x
\end{eqnarray*}
for every $t\in \mathbb{R}$. This family of maps is in fact a
one-parameter group of diffeomorphism of $M$, in other words, we
have:
$$
    \varphi_{t_1+t_2} = \varphi_{t_1}\circ\varphi_{t_2}
$$
for every $t_1, t_2 \in \mathbb{R}$. Hence, there is a smooth vector
field $X^*$ on $M$ given by:
$$
    X^*_x = \frac{d}{dt}\Big|_{t=0} (\exp(tX)x).
$$
Also, it is easy to check that the global flow of $X^*$ is given by
$(\varphi_t)_t$. Furthermore, since the Lie group $G$ acts locally
freely, the condition $X^*_x = 0$ for some $x\in M$ implies $X=0$;
otherwise the subgroup $(\exp(tX))_t$ would be nondiscrete and
contained in $G_x$ for some $x\in M$.

From the above remarks it follows that the map:
\begin{eqnarray*}
        M\times \mathfrak{g} &\rightarrow& TM \\
        (x,X) &\mapsto& X^*_x
\end{eqnarray*}
is a smooth vector bundle inclusion which thus defines a subbundle
$T\mathcal{O}$ of $TM$.

On the other hand, by using the results in \cite{KobNomizu96}, the
following relation holds for every $X, Y \in \mathfrak{g}$:
$$
        [X^*,Y^*] = - [X,Y]^*
$$
From this it is easy to conclude that the smooth sections of
$T\mathcal{O}$ are closed under the Lie brackets of smooth vector
fields. By Frobenius theorem (see \cite{Warner83}) the vector
subbundle $T\mathcal{O}$ induces a smooth foliation whose leaves
have the fibers of $T\mathcal{O}$ as tangent spaces. Since $G$ is
connected it is generated by the set $\exp{\mathfrak{g}}$, and so
one can conclude that the leaves of such foliation are precisely the
$G$-orbits.
\end{proof}

In the proof of the previous result it is shown that the tangent
bundle of the foliation by $G$-orbits is $T\mathcal{O}$. Whenever
$G$ acts locally freely we will use $T\mathcal{O}$ to denote such
tangent bundle.

We will now consider the case where $G$ acts locally freely
preserving a Riemannian metric on $M$.

\begin{theorem}\label{G-fol-Riem-fol}
If $G$ acts locally freely on $M$ preserving a Riemannian metric
$h$, then the $G$-orbits define a smooth Riemannian foliation for
which $h$ is a compatible bundle-like metric.
\end{theorem}
\begin{proof}
By Theorem \ref{bundle-like-vs-Riem-fol} it is enough to show that
$h$ is bundle-like with respect to the foliation by $G$-orbits
given by Proposition \ref{G-orbits-fol}.

Choose $X$ and $Y$ foliate vector fields perpendicular to the
$G$-orbits. We need to prove that $v(h(X,Y)) = 0$, for every $v\in
T\mathcal{O}$. By the proof of Proposition \ref{G-orbits-fol} there
exists $Z\in \mathfrak{g}$, the Lie algebra of $G$, such that $Z^*_x
= v$, where $x$ is the basepoint of $v$. Hence, it suffices to prove
that $Z^*(h(X,Y)) = 0$ for every $Z \in \mathfrak{g}$.

For any $Z^*$ as above, we denote with $L_{Z^*}$ the Lie derivative
with respect to $Z^*$ and refer to \cite{KobNomizu96} for the definition. In
fact, from \cite{KobNomizu96} it follows that $L_{Z^*}$ when applied to $h$
yields a bilinear form that satisfies:
\begin{equation}\label{Lie-h}
    (L_{Z^*}h)(X,Y) = Z^*(h(X,Y)) - h([Z^*,X],Y) - h(X,[Z^*,Y]).
\end{equation}
On the other hand, since the one-parameter group $(\exp(tX))_t$ acts
by isometries on $(M,h)$, i.e. preserving $h$, it follows that $Z^*$
is a Killing field for $h$ and so it satisfies:
\begin{equation}\label{Z-Killing}
    (L_{Z^*}h)(X,Y) = 0,
\end{equation}
we refer to \cite{KobNomizu96} for this fact and the definitions involved.
From equations (\ref{Lie-h}) and (\ref{Z-Killing}) we obtain:
$$
    Z^*(h(X,Y)) = h([Z^*,X],Y) + h(X,[Z^*,Y]).
$$
Then we observe that, since $X$ and $Y$ are foliate, the vector
fields $[Z^*,X]$ and $[Z^*,Y]$ are tangent to the $G$-orbits, and so
the terms on the right-hand side of the last equation vanish since
$X$ and $Y$ are also perpendicular to the $G$-orbits. This shows
that $Z^*(h(X,Y)) = 0$ thus concluding the proof.
\end{proof}

From the last result and Theorem \ref{transverse-tot-geod} we obtain
the following consequence.

\begin{theorem}\label{G-orbits-transverse-tot-geod}
If $G$ acts locally freely on $M$ preserving a Riemannian metric $h$
and $\gamma$ is a geodesic (with respect to $h$) perpendicular at
some point to a $G$-orbit, then $\gamma$ intersects every $G$-orbit
perpendicularly.
\end{theorem}

The previous result and the remarks following Theorem
\ref{transverse-tot-geod} allows us to say that, from a geometric
point of view, every locally free action of a group $G$ preserving a
Riemannian metric $h$ defines (through its orbits) a foliation which
is transversely totally geodesic.

\section{Lagrangian foliations associated with a Reinhardt domain}
We now proceed to study the geometry of Reinhardt domains. For this
we will obtain some properties of its Bergman metric and apply the
foliation theory considered in the previous sections. As before, in
this section $D\subset \C^n$ denotes a bounded logarithmically
convex complete Reinhardt domain centered at the origin.

Using the monomial orthonormal base $\{e_p\}_{p \in
\Z_+^n}$ of $\Aa^2_{\mu}(D)$, mentioned in Section \ref{se:Reinh-Toepl}, we have obviously

\begin{lemma}\label{le:Bergman-Reinhardt}
The Bergman kernel $K_D$ of the domain $D$ admits the following representation
\begin{equation*}
    K_D(z, \zeta) = (2\pi)^{-n} \sum_{p\in\Z^n_+} \alpha_p^2\, z^p \bar{\zeta}^p,
\end{equation*}
where the coefficients $\alpha_p$, $p\in\Z^n_+$, are given by
(\ref{eq:alpha_p}). In particular, the function $K_D(z,z)$ depends
only on $r$.
\end{lemma}

In this section we will use the polar coordinates $z_k=r_k t_k=r_k e^{i\theta_k}$, $k=1,...,n$, for points $z=(z_1,...,z_n) \in D$.

\begin{theorem}\label{th:metric-Reinhardt}
Let $ds^2_D$ be the Bergman metric of $D$ considered as a Hermitian metric and $h_D = \re(ds^2_D)$ the associated Riemannian metric. Then:
    $$
        h_D = \sum_{k,l=1}^n F_{kl}(r)(dr_k\otimes dr_l  +
                    r_k r_l d\theta_k \otimes d\theta_l),
    $$
where the functions $F_{kl}$ are given by:
    $$
    F_{kl}(r) = \frac{1}{4}
        \left(\frac{\partial^2}{\partial r_k \partial r_l}
        + \frac{\delta_{kl}}{r_k}\frac{\partial}{\partial r_k}
        \right) \log K_D(z,z),
    $$
and depend only on $r$.
\end{theorem}

\begin{proof}
For the Bergman kernel $K_D$, the associated Bergman metric
considered as a Hermitian metric is given by:
    $$
ds^2_D = \sum_{k,l=1}^n\frac{\partial^2 \log K_D(z,z)}{\partial z^k \partial \bar{z}^l} dz^k \otimes d\bar{z}^l.
    $$
Let $F(r) = F(z) = \log K_D(z,z)$, which by Lemma
\ref{le:Bergman-Reinhardt} depends only on $r$. Then a
straightforward computation shows that:
    $$
\frac{\partial^2 F}{\partial z_k \partial \bar{z}_l}(z) =
\frac{1}{4} \left(\frac{\bar{z}_k z_l}{r_k r_l}
\frac{\partial^2 F}{\partial r_k \partial r_l}(z)
+ \frac{\delta_{kl}}{r_k}\frac{\partial F}{\partial r_k}(z)
\right).
    $$
The required identity is then obtained by replacing these
expressions into that of $ds^2_D$, using the relations $z_k = r_k
e^{i\theta_k}$ and computing the real part of the expression thus
obtained.
\end{proof}

Consider the following action of the $n$-dimensional torus $\mathbb{T}^n$ on $D$
\begin{eqnarray*}
    \mathbb{T}^n \times D &\rightarrow& D \\
    (t, z) &\mapsto& t z,
\end{eqnarray*}
which being biholomorphic yields the following immediate
consequence.

\begin{theorem}\label{torus-isometries}
Let $h_D$ be the Riemannian metric of $D$ defined by its Bergman
metric. Then $\mathbb{T}^n$ acts isometrically on $(D,h_D)$.
\end{theorem}

Note that the action of $\mathbb{T}^n$ is not locally free at all
points of an $n$-dimensional Reinhardt domain, but it is almost so
as the following obvious result states. We recall that in a measure
space, a subset is called conull if its complement has zero measure.

\begin{lemma}\label{Reinhardt-loc-free}
For $D$ as before, the set:
$$
    \widehat{D} = \{ z\in D : \ z_k \neq 0 \mbox{ for every } k=1,\dots,n \}
$$
is the set of points whose stabilizers with respect to the action of
$\mathbb{T}^n$ are discrete. Furthermore, $\widehat{D}$ is an open
conull subset of $D$ on which $\mathbb{T}^n$ acts freely.
\end{lemma}

As a consequence of Theorems \ref{G-fol-Riem-fol} and
\ref{torus-isometries} and Lemma \ref{Reinhardt-loc-free} we obtain
the following.

\begin{theorem}\label{torus-thm1}
Let $D$ be as before, $\widehat{D}$ the subset of $D$ defined in
Lemma \ref{Reinhardt-loc-free} and $h_D$ the Riemannian metric
defined by the Bergman metric of $D$. Then, the
$\mathbb{T}^n$-orbits in $\widehat{D}$ define a Riemannian foliation
$\Oa$ for which $h$ is a compatible bundle-like metric.
\end{theorem}

Given such result we now obtain the following statement which makes
use of Theorem \ref{G-orbits-transverse-tot-geod} as well.

\begin{theorem}\label{torus-thm2}
Let $D$ be as before, $\widehat{D}$ the subset of $D$ defined in
Lemma \ref{Reinhardt-loc-free} and $h_D$ the Riemannian metric
defined by the Bergman metric of $D$. If $\gamma$ is a geodesic in
$\widehat{D}$ (with respect to $h$) perpendicular at some point to a
$\mathbb{T}^n$-orbit, then $\gamma$ intersects every
$\mathbb{T}^n$-orbit perpendicularly.
\end{theorem}

    We now prove that the Riemannian foliation $\Oa$ obtained in the
    previous result is Lagrangian.

    \begin{theorem}\label{th:O-Lag-Riem}
    Let $D$ be as before, $\widehat{D}$ the subset of $D$ defined in
    Lemma \ref{Reinhardt-loc-free}, $ds^2_D$ the Bergman metric of
    $D$ as a
    Hermitian metric and $\Oa$ the Riemannian foliation of $\T^n$-orbits
    in $\widehat{D}$. Then $\Oa$ is Lagrangian with respect to the
    Riemannian metric $h_D = \re(ds^2_D)$,
    in other words, the $\T^n$-orbits in $\widehat{D}$ are Lagrangian
    with respect to $h_D$.
    \end{theorem}
    \begin{proof}
    We need to prove that $T_z \Oa$ and $iT_z \Oa$ are perpendicular
    with respect to the Riemannian metric $h_D = \re(ds^2_D)$ at every
    $z\in \widehat{D}$. Since such condition is invariant under the
    $\T^n$-action we can assume that $z=x \in \R_+^n$.

    We observe that for every $x\in \widehat{D}$ we have $T_x\Oa =
    i\R^n$. Hence the result follows by using Theorem
    \ref{th:metric-Reinhardt} together with the fact that $i\R^n$ and
    $\R^n$ are perpendicular with respect to the elements $dr_k\otimes
    dr_l  + r_k r_l d\theta_k \otimes d\theta_l$ for every $k,l$.
    \end{proof}

    We now prove that the normal bundle to $\Oa$ is integrable.

    \begin{theorem}
    Let $D$ be as before, $\widehat{D}$ the subset of $D$ defined in
    Lemma \ref{Reinhardt-loc-free} and $h_D$ the Riemannian metric
    defined by the Bergman metric of $D$. If we denote with $T\Oa^\perp$
    the vector
    subbundle of $T\widehat{D}$ of tangent vectors perpendicular to
    $\Oa$, then $T\Oa^\perp$ is integrable to a foliation $\Pa$.
    Furthermore, $\Pa$ is a Lagrangian totally geodesic foliation of
    $\widehat{D}$.
    \end{theorem}
    \begin{proof}
    If we let $M_0 = D\cap \R_+^n$, then by the proof of Theorem
    \ref{th:O-Lag-Riem} the tangent bundle to $M_0$ coincides with
    $T\Oa^\perp$ restricted to $M_0$, and so $M_0$ is an integral
    submanifold of $T\Oa^\perp$. Since $T\Oa^\perp$ is invariant under
    the $\T^n$-action and such action preserves the metric, it follows
    that for every $t\in \T^n$ the manifold:
    $$
            M_t = tM_0
    $$
    is an integral submanifold of $T\Oa^\perp$, thus showing the
    integrability of such bundle to some foliation $\Pa$.

    By Lemma \ref{le:Bergman-Reinhardt} we have $T\Pa = T\Oa^\perp =
    iT\Oa$ which implies that $\Pa$ is Lagrangian. Finally $\Pa$ is
    totally geodesic by Theorem \ref{torus-thm2}.
    \end{proof}

We now state the following easy corollary of the previous
discussion.

\begin{corollary}\label{cor:framing}
The sets of vector fields:
$$
    \left(\frac{\partial}{\partial \theta_k}\right)_{k=1}^n \quad
    \mbox{ and } \quad
    \left(\frac{\partial}{\partial r_k}\right)_{k=1}^n,
$$
define global framings for the bundles $T\Oa$ and $T\Oa^\perp =
T\Pa$, respectively, on $\widehat{D}$.
\end{corollary}

\section{The unit ball}
An important class of domains in complex analysis is given by those
which are bounded and symmetric. The next result shows that each
irreducible bounded symmetric domain which is also Reinhardt has to
be a unit ball. As usual, we will denote by $\mathbb{B}^n$ the unit
ball in $\mathbb{C}^n$.

\begin{theorem}\label{bdd-Reinhardt}
Let $D$ be an irreducible bounded symmetric domain. Then $D$ is also
a Reinhardt domain if and only if $D = \mathbb{B}^n$ for some $n \in
\Z_+$.
\end{theorem}
\begin{proof}
First, the unit ball centered at the origin in a complex vector
space is obviously a Reinhardt domain. Conversely, let us assume
that $D$ is an irreducible bounded symmetric domain which is also
Reinhardt. We show that it is a unit ball centered at the origin of
some complex vector space. For this we use Cartan's classification
of irreducible bounded symmetric domains and the description of
their biholomorphisms as found in \cite{Helgason01}. We present the
needed basic properties in Table~\ref{table-bdd}, which recollects
some of the information found in Table V in page 518 from
\cite{Helgason01}. Every irreducible bounded symmetric domain $D$ in
Table~\ref{table-bdd} is identified by its type in the first column
(following the notation from \cite{Helgason01}) and is explicitly
given as the quotient $G_0/K$ for the groups in the second and third
column. The group $G_0$ is, up to a finite covering, the group of
biholomorphisms of $D$ and $K$ is the subgroup of $G_0$ consisting
of those transformations that fix the origin. For the exceptional
bounded symmetric domains of type \textbf{EIII} and \textbf{EVII} we
write down the Lie algebras of the corresponding groups, which is
enough for our purposes; again, we follow here the notation from
\cite{Helgason01} to identify real forms of exceptional complex Lie
algebras. The last two columns permit us to compare the complex
dimension of $D$ and the dimension of a maximal torus $T$ in $K$.
This last dimension is well known from the basic properties of the
compact groups $K$ that appear in Table~\ref{table-bdd}. We recall
from the basic theory of symmetric spaces that the universal
covering of the group $G_0$ completely determines the bounded
symmetric domain: in other words, two bounded symmetric domains
whose corresponding groups $G_0$ in Table~\ref{table-bdd} have the
same universal covering group are biholomorphic. Through out
Table~\ref{table-bdd}, the symbols $p$, $q$ and $n$ are assumed to
be positive integers. The additional conditions on the types
\textbf{BDI(2,q)} and \textbf{DIII} are required for the
corresponding quotient $G_0/K$ to actually define an irreducible
bounded symmetric domain.

\begin{table}[ht]
\begin{center}
\renewcommand{\arraystretch}{1.3}
\caption{Irreducible bounded symmetric domains}\label{table-bdd}
\begin{tabular}{ l l l l l }
  \hline
  $D$               & $G_0$                     & $K$                   & $\dim_\mathbb{C}(D)$  &  $\dim(T)$ \\ \hline
  \textbf{AIII}     & $SU(p,q)$                 & $S(U(p)\times U(q))$  &  $pq$  &  $p+q-1$                 \\
  \textbf{BDI(2,q)} ($q \neq 2$) & $SO_0(2,q)$               & $SO(2)\times SO(q)$   &  $q$   &  $\left[ \frac{q}{2} \right] + 1$  \\
  \textbf{DIII} ($n\geq 2$)     & $SO^*(2n)$                & $U(n)$                &  $\frac{n(n-1)}{2}$   &  $n$      \\
  \textbf{CI}       & $Sp(n,\mathbb{R})$        & $U(n)$                &  $\frac{n(n+1)}{2}$   &  $n$      \\
  \textbf{EIII}     & $\mathfrak{e}_{6(-14)}$   & $\mathfrak{so}(10)\oplus\mathbb{R}$  & $16$   &  $6$      \\
  \textbf{EVII}     & $\mathfrak{e}_{7(-25)}$   & $\mathfrak{e}_6\oplus\mathbb{R}$  & $27$      &  $7$      \\
  \hline
\end{tabular}
\end{center}
\end{table}

For $D$ in Table~\ref{table-bdd} to be a Reinhardt domain, we
clearly have as a necessary condition the inequality:
\begin{equation}\label{Reinhardt-ineq}
\dim(T) \geq \dim_\mathbb{C}(D).
\end{equation}
Let us now consider the cases where this might occur in
Table~\ref{table-bdd}.

\begin{description}
    \item[AIII] The condition (\ref{Reinhardt-ineq}) holds if and
    only if $\min(p,q) = 1$, which clearly corresponds to the unit ball of
    dimension $\max(p,q)$.

    \item[BDI(2,q)] In this case the condition (\ref{Reinhardt-ineq}) holds
    if and only if $q=1$. This corresponds to the bounded symmetric
    domain whose group of biholomorphisms is, up to a finite
    covering, $SO_0(2,1)$. Since the Lie algebras
    $\mathfrak{so}(2,1)$ and $\mathfrak{su}(1,1)$ are isomorphic,
    the bounded symmetric domain of type \textbf{BDI(2,1)} is the unit
    disc in the complex plane.

    \item[DIII] In this case the condition (\ref{Reinhardt-ineq}) holds
    only for $n = 2$ or $3$. The Lie algebras of the corresponding
    groups $G_0$ are $\mathfrak{so}^*(4)$ and $\mathfrak{so}^*(6)$.
    There are well known isomorphisms $\mathfrak{so}^*(4) \cong
    \mathfrak{su}(2) \times \mathfrak{su}(1,1)$ and
    $\mathfrak{so}^*(6) \cong \mathfrak{su}(3,1)$ (see
    \cite{Helgason01}). We also recall that $\mathfrak{u}(n) \cong
    \mathfrak{su}(n) \oplus \mathbb{R}$, for every $n$. Hence, we
    conclude that type \textbf{DIII} for $n = 2$ and $3$ defines
    the unit disc in the complex plane and the unit ball in
    $\mathbb{C}^3$, respectively.

    \item[CI] In this case the condition (\ref{Reinhardt-ineq}) holds only for $n=1$, which
    yields the unit disk with an argument as above using the fact
    that $\mathfrak{sp}(1,\R)$ is isomorphic to
    $\mathfrak{su}(1,1)$ (see \cite{Helgason01}).

    \item[EIII, EVII] A simple inspection shows
    that in these cases the condition (\ref{Reinhardt-ineq}) cannot
    hold.
\end{description}
This completes the proof of Theorem \ref{bdd-Reinhardt}.
\end{proof}

Now the results of the previous sections lead directly to the
following statements:


\begin{enumerate}
  \item On the subset $\widehat{\B}^n$ the $\T^n$-action defines a
  Lagrangian foliation $\Oa$.
  \item The orthogonal complement $T\Oa^\perp$ is integrable in
  $\widehat{\B}^n$ to a foliation totally geodesic Lagrangian
  foliation $\Pa$.
  \item The pair of foliations $\Oa$ and $\Pa$ define the polar
  coordinates in $\widehat{\B}^n$, which in turn yields
  the commutative $C^*$- algebra of Toeplitz operators whose symbols are constant
  on the leaves of $\Oa$.
\end{enumerate}

In what follows we will normalize the (Hermitian) Bergman metric on
the unit ball to the following expression:
$$
        ds^2 = \frac{4}{1-\sum_{k=1}^n |z_k|^2}
        \left(\sum_{k=1}^n dz^k\otimes d\overline{z}^k +
        \sum_{k,l=1}^n \frac{\overline{z}_k z_l \,dz^k\otimes
        d\overline{z}^l}{1-\sum_{k=1}^n |z_k|^2}\right).
$$
which differs from the usual Bergman metric as considered in the
proof of Theorem \ref{th:metric-Reinhardt} by a factor of $(n+1)/4$.
The advantage of this normalization is that the sectional curvature
varies in the interval $[-1,-1/4]$, while with the metric as defined
in the proof of Theorem \ref{th:metric-Reinhardt} the sectional
curvature varies in the interval $[-4/(n+1),-1/(n+1)]$.

We will now compute some values of the second fundamental form for
the foliation $\Oa$ of the unit ball. First, we recall the notion of
complex geodesic and some of its properties.

\begin{definition}
A complex geodesic in $\B^n$ is a biholomorphic map $\varphi : \D
\rightarrow \D'$ where $\D$ is the unit disc and $\D'=\B^n\cap L$
for some complex affine line $L$ in $\C^n$.
\end{definition}

It is well known that complex geodesics are always totally geodesic
maps. Furthermore, the images of complex geodesics are precisely the
closed totally geodesic complex submanifolds of (complex) dimension
$1$ in $\B^n$ (see \cite{Goldman99}).

The next result shows that some of the orbits of the $\T^n$-action
on the unit ball integrate the vector fields of the framing
$\left(\frac{\partial}{\partial\theta_k}\right)_{k=1}^n$ from
Corollary \ref{cor:framing}. Its proof is a straightforward
computation.

\begin{lemma}\label{le:geodesics-O}
For every $k,1=1,\dots,n$ with $k\neq l$, the curves:
\begin{eqnarray*}
  \gamma_{z,k}(s) &=& (z_1,\dots,z_{k-1},e^{is}z_k,z_{k+1},\dots,z_n) \\
  \gamma_{z,kl}(s) &=& (z_1,\dots,z_{k-1},e^{is}z_k,z_{k+1},\dots,z_{l-1},e^{is}z_l,z_{l+1},\dots,z_n)
\end{eqnarray*}
are integral curves of the vector fields
$\frac{\partial}{\partial\theta_k}$ and
$\frac{\partial}{\partial\theta_k} +
\frac{\partial}{\partial\theta_l}$, respectively.
\end{lemma}
\begin{proof}
By the definition of polar coordinates it is clear that the flows
that integrate $\frac{\partial}{\partial\theta_k}$ and
$\frac{\partial}{\partial\theta_k} +
\frac{\partial}{\partial\theta_l}$ are given by:
\begin{eqnarray*}
    z &\mapsto& (z_1,\dots,z_{k-1},e^{is}z_k,z_{k+1},\dots,z_n), \\
    z &\mapsto& (z_1,\dots,z_{k-1},e^{is}z_k,z_{k+1},\dots,z_{l-1},e^{is}z_l,z_{l+1},\dots,z_n)
\end{eqnarray*}
from which the conclusion is clear.
\end{proof}

Let us define the following vector fields on $\widehat{\B}^n$:
\begin{eqnarray*}
  Q_k &=& \II\left(\frac{\partial}{\partial\theta_k},\frac{\partial}{\partial\theta_k}\right) \\
  Q_{kl} &=& \II\left(\frac{\partial}{\partial\theta_k} +
\frac{\partial}{\partial\theta_l},\frac{\partial}{\partial\theta_k}
+ \frac{\partial}{\partial\theta_l}\right),
\end{eqnarray*}
then, by Proposition \ref{prop:II-quadratic} and Corollary
\ref{cor:framing}, such vector fields completely determine the
second fundamental form $\II$. We will compute $Q_k$ and $Q_{kl}$
using the curves defined in Lemma \ref{le:geodesics-O}. To achieve
this, for every $z\in \widehat{\B}^n$ we define the following
complex geodesics:
\begin{eqnarray*}
  \phi_{z,k}(w) &=& (z_1,\dots,z_{k-1},R_k w,z_{k+1},\dots,z_n) \\
  \phi_{z,kl}(w) &=& (z_1,\dots,z_{k-1},R_{kl}w,z_{k+1},\dots,z_{l-1},\frac{R_{kl}z_l}{z_k}w,z_{l+1},\dots,z_n)
\end{eqnarray*}
where $k,l=1,\dots,n$ with $k\neq l$ and:
\begin{eqnarray*}
  R_k &=& \sqrt{1-\sum_{j\neq k} |z_j|^2} \\
  R_{kl} &=& \frac{|z_k|\sqrt{1-\sum_{j\neq k,l} |z_j|^2}}{\sqrt{|z_k|^2 +
  |z_l|^2}}.
\end{eqnarray*}
Then we have the following easy to prove result.

\begin{lemma}\label{le:complex-geodesics-curves}
For every $z\in \widehat{\B}^n$ the complex geodesics $\phi_{z,k}$,
$\phi_{z,kl}$ satisfy:
\begin{enumerate}
  \item $\phi_{z,k}(z_k/R_k) = \phi_{z,kl}(z_k/R_{kl}) = z$ for
  every $k,l = 1,\dots,n$ with $k\neq l$,
  \item $\gamma_{z,k}(\R)\subset\phi_{z,k}(\D)$ and
  $\gamma_{z,kl}(\R)\subset\phi_{z,kl}(\D)$,
\end{enumerate}
in other words, they pass through $z$ and contain the curves from
Lemma \ref{le:geodesics-O} with the same indices.
\end{lemma}

We now use the above to compute the value of the vector fields $Q_k$
and $Q_{kl}$.

\begin{lemma}\label{le:Q-values}
For every $z\in \widehat{\B}^n$ and $k,l=1,\dots,n$ with $k\neq l$
we have the following relations:
\begin{enumerate}
  \item $Q_k(z) = \gamma_{z,k}''(0)$ and $Q_{kl}(z) =
  \gamma_{z,kl}''(0)$, where the acceleration is computed for the
  complex hyperbolic geometry of $\B^n$,
  \item $\gamma_{z,k}''(s) \in \R i \gamma_{z,k}'(s)$ and
  $\gamma_{z,kl}''(s) \in \R i \gamma_{z,kl}'(s)$ for every $s\in
  \R$; in particular:
  $$
        Q_k(z) \in \R \frac{\partial}{\partial r_k}\Big|_z,
        \quad
        Q_{kl}(z) \in \R \left( \frac{\partial}{\partial r_k}\Big|_z
                + \frac{\partial}{\partial r_l}\Big|_z\right),
  $$
  \item the norms of $Q_k$ and $Q_{kl}$ are given by:
  $$
        \|Q_k(z)\| = C_k(z) \|\gamma_{z,k}'(0)\|^2,
        \quad
        \|Q_{kl}(z)\| = C_{kl}(z) \|\gamma_{z,kl}'(0)\|^2,
  $$
  where $C_k(z)$ and $C_{kl}(z)$ are the geodesic curvatures of
  $\gamma_{z,k}$ and $\gamma_{z,kl}$, respectively, considered as
  curves in the images of the complex geodesics $\phi_{z,k}$ and
  $\phi_{z,kl}$, respectively, endowed with the metric inherited from
  $\B^n$.
\end{enumerate}
\end{lemma}
\begin{proof}
First we observe that in $\widehat{\B}^n$ the leaves of the
foliation $\Oa$ by $\T^n$-orbits are diffeomorphic to $\T^n$ under
the action map. In particular, with respect to such diffeomorphisms,
the metric of $\B^n$ restricted to any such $\T^n$-orbit is left
invariant. For such metrics on abelian Lie groups it is well known
that the geodesics are precisely the one parameter subgroups and
their translations (see \cite{Helgason01}). Since the curves
$\gamma_{z,k}$, $\gamma_{z,kl}$ correspond to one parameter groups
in $\T^n$ it follows that they define geodesics in the leaf of $\Oa$
through $z$. Then, by well known results on the geometry of
Riemannian submanifolds (see \cite{Oneill83}) it follows that the
accelerations $\gamma_{z,k}''$, $\gamma_{z,kl}''$ as computed in
$\B^n$ are everywhere perpendicular to the leaves of $\Oa$, in other
words they are everywhere horizontal.

By the remarks in Section \ref{se:Extrinsic-fol} and the definition
of $Q_k$ and $Q_{kl}$ we have:
\begin{eqnarray*}
  Q_k(z) &=& \Ha(\widehat{\nabla}_{\gamma_{z,k}'(0)}\gamma_{z,k}') = \Ha(\gamma_{z,k}''(0)) = \gamma_{z,k}''(0) \\
  Q_{kl}(z) &=&
  \Ha(\widehat{\nabla}_{\gamma_{z,kl}'(0)}\gamma_{z,kl}') = \Ha(\gamma_{z,kl}''(0)) = \gamma_{z,kl}''(0),
\end{eqnarray*}
for $\widehat{\nabla}$ the connection of $\B^n$, and where the last
identities follow from the remarks in the previous paragraph. This
proves (1).

Next observe that since the curves $\gamma_{z,k}$, $\gamma_{z,kl}$
are geodesics in some leaf of $\Oa$ it follows that they are up to a
constant parameterized by arc-length. On the other hand, by Lemma
\ref{le:complex-geodesics-curves} the curves $\gamma_{z,k}$,
$\gamma_{z,kl}$ lie in complex geodesics which, as we observed
before, define totally geodesic submanifolds of $\B^n$. In
particular, their accelerations $\gamma_{z,k}''$, $\gamma_{z,kl}''$
as computed in $\B^n$ are the same as computed in the (images) of
the complex geodesics that contain them. Complex geodesics are
isometric to the unit disk, and for the latter any curve $\gamma$
which is parameterized up to a constant by arc length satisfies
$\gamma''(s) \in \R i \gamma'(s)$. This implies the first part of
(2). For the second part of (2) it is enough to note that:
\begin{eqnarray*}
  i\gamma_{z,k}'(s) &=& i\frac{\partial}{\partial \theta_k}\Big|_{\gamma_{z,k}(s)} \in \R \frac{\partial}{\partial r_k}\Big|_{\gamma_{z,k}(s)} \\
  i\gamma_{z,kl}'(s) &=& i\left(\frac{\partial}{\partial \theta_k}\Big|_{\gamma_{z,kl}(s)}
    + \frac{\partial}{\partial \theta_l}\Big|_{\gamma_{z,kl}(s)}
    \right)
   \in \R \left(\frac{\partial}{\partial r_k}\Big|_{\gamma_{z,kl}(s)}
    + \frac{\partial}{\partial r_l}\Big|_{\gamma_{z,kl}(s)}
    \right),
\end{eqnarray*}
and so we obtain (2) by applying (1).

By the definition of the geodesic curvature (see
\cite{GrudQuirogaVasil05} and its references) we have:
\begin{eqnarray*}
  C_k(z) &=& \frac{\|\gamma_{z,k}''(0)\|}{\|\gamma_{z,k}'(0)\|^2} \\
  C_{kl}(z) &=& \frac{\|\gamma_{z,kl}''(0)\|}{\|\gamma_{z,kl}'(0)\|^2},
\end{eqnarray*}
where we have used the fact that the norm of vectors and the
acceleration of curves in a (image of a) complex geodesic in $\B^n$
computed in $\B^n$ or the complex geodesic yield the same result.
Given the above identities, (3) follows from (1).
\end{proof}

The next result computes specific values for the second fundamental
form of the foliation $\Oa$ for the unit ball. By Proposition
\ref{prop:II-quadratic} such values completely determine the second
fundamental form. We observe that in the first two parts of the
statement we obtain a very explicit expression for the second
fundamental form of the foliation $\Oa$. Note that by part (3) of
Lemma \ref{le:Q-values}, the geodesic curvatures $C_k$ and $C_{kl}$
correspond to the norms of the values of $Q_k$ and $Q_{kl}$,
respectively, normalized so that they only depend on the direction
of $\frac{\partial}{\partial\theta_k}$ and
$\frac{\partial}{\partial\theta_k} +
\frac{\partial}{\partial\theta_l}$, respectively. In view of this,
the last part of the statement allows us to understand the
asymptotic behavior of the curvature of the leaves of $\Oa$ as they
move towards the origin or the boundary of $\widehat{\B}^n$. We
observe as well that this result generalizes our geometric
description of the elliptic model case in the unit disk found in
\cite{GrudQuirogaVasil05}.

\begin{theorem}\label{th:ball-curv}
For every $z\in \widehat{\B}^n$, let $r = (r_1,\dots,r_n) =
(|z_1|,\dots,|z_n|)$, and consider the curves $\gamma_{z,k}$,
$\gamma_{z,kl}$ and the complex geodesics $\phi_{z,k}$,
$\phi_{z,kl}$ defined above. Then:
\begin{enumerate}
  \item The vector fields $Q_k$ and $Q_{kl}$ are given by:
  \begin{eqnarray*}
    Q_k(z) &=& -C_k(z) \left\|\frac{\partial}{\partial \theta_k}\Big|_z\right\|^2
                \left\|\frac{\partial}{\partial r_k} \Big|_z\right\|^{-1}
                \left(\frac{\partial}{\partial r_k} \Big|_z\right) \\
    Q_{kl}(z) &=& -C_{kl}(z) \left\|\frac{\partial}{\partial \theta_k}\Big|_z + \frac{\partial}{\partial \theta_l}\Big|_z
                \right\|^2
                \left\|\frac{\partial}{\partial r_k} \Big|_z
                    + \frac{\partial}{\partial r_l} \Big|_z \right\|^{-1}
                \left( \frac{\partial}{\partial r_k} \Big|_z +
                \frac{\partial}{\partial r_l} \Big|_z  \right)
  \end{eqnarray*}
  \item The geodesic curvatures $C_k(z)$ and $C_{kl}(z)$ at $z$ defined in
  Lemma \ref{le:Q-values} are given by:
  \begin{eqnarray*}
    C_k(z) &=& \frac{r_k^2 + \left(1-\sum_{j\neq k} r_j^2\right)}{2r_k\sqrt{1-\sum_{j\neq k} r_j^2}} \\
    C_{kl}(z) &=& \frac{r_k^2 + r_l^2 + \left(1-\sum_{j\neq k,l} r_j^2\right)}{2\sqrt{r_k^2+r_l^2}\sqrt{1-\sum_{j\neq k,l}
    r_j^2}},
  \end{eqnarray*}
  in particular, such geodesic curvatures lie in the interval
  $(1,+\infty)$ and achieve all values therein.
  \item The geodesic curvatures $C_k(z)$ and $C_{kl}(z)$ have the
  following asymptotic behavior:
  \begin{eqnarray*}
  C_k(z), C_{kl}(z) &\rightarrow& +\infty, \quad \mbox{ as } |z| \rightarrow 0,  \\
  C_k(z) &\rightarrow& 1, \quad
                      \mbox{ as } z \rightarrow u, \\
  C_{kl}(z) &\rightarrow& 1, \quad
                      \mbox{ as } z \rightarrow v.
  \end{eqnarray*}
  for any $u,v \in \partial\B^n$ such that $u_k \neq 0$ and $|v_k|^2
  + |v_l|^2 \neq 0$, respectively.
\end{enumerate}
\end{theorem}
\begin{proof}
Up to a sign, (1) essentially follows from (2) and (3) in Lemma
\ref{le:Q-values}. The negative sign comes from the fact that, in
the proof of Lemma \ref{le:Q-values}, the accelerations of
$\gamma_{z,k}$, $\gamma_{z,kl}$ point towards the origin in the
complex geodesics that contain them and the vector fields
$$
    \frac{\partial}{\partial r_k}, \quad
    \frac{\partial}{\partial r_k} + \frac{\partial}{\partial r_l}
$$
point away from the origin.

To prove (2), let $\phi_{z,k}$, $\phi_{z,kl}$ be the complex
geodesics considered before. Then the inverse images of the curves
$\gamma_{z,k}$, $\gamma_{z,kl}$ with respect to such maps are easily
seen to be circles in $\D$ centered at the origin with Euclidean
radius:
$$
        s_k = \frac{r_k}{R_k} = \frac{r_k}{\sqrt{1 - \sum_{j\neq
        k}r_j^2}} \quad \mbox{and} \quad
        s_{kl} = \frac{r_k}{R_{kl}} = \frac{\sqrt{r_k^2 +
              r_l^2}}{\sqrt{1-\sum_{j\neq k,l} r_j^2}},
$$
respectively. Next, we observe that the geodesic curvature $C(s)$ of
the circle with Euclidean radius $s$ in the unit disk $\D$ with the
metric:
$$
          \frac{4(dx^2+dy^2)}{(1-(x^2+y^2))^2}
$$
is given by the formula:
$$
    C(s) = \frac{1+s^2}{2s}.
$$
This follows from two facts found in \cite{Goldman99}. The
hyperbolic radius $\rho$ of a circle centered at the origin
satisfies $\cosh^2(\rho/2) = 1/(1-s^2)$, where $s$ is the Euclidean
radius. And the geodesic curvature of such a circle is given by
$\coth(\rho)$. The first can be deduce from the expression for the
hyperbolic distance found in subsection 1.4.1 of \cite{Goldman99}
and the second is stated in subsection 1.4.2 of the same reference.

Given the above formula for $C(s)$ a simple substitution provides
the required expressions for $C_k$ and $C_{kl}$. Finally, (3) is a
consequence of these expressions.
\end{proof}

We recall that the velocity and acceleration of curves in a manifold
do not change when we renormalize the metric of the manifold by a
constant multiple. More generally, the second fundamental form of a
submanifold does not change either by such renormalizations (see
\cite{Oneill83}). However, the geodesic curvatures as defined above
involve the metric and so they are rescaled when we renormalize the
metric by a constant.

In particular, for the Bergman metric on $\B^n$ (i.e. without
normalizing to have sectional curvature in the interval $[-1,
-1/4]$) which is given by
$$
        ds^2_{\B^n} = \frac{n+1}{1-\sum_{k=1}^n |z_k|^2}
        \left(\sum_{k=1}^n dz^k\otimes d\overline{z}^k +
        \sum_{k,l=1}^n \frac{\overline{z}_k z_l \,dz^k\otimes
        d\overline{z}^l}{1-\sum_{k=1}^n |z_k|^2}\right).
$$
the tangent vectors $\gamma_{z,k}'(0)$, $\gamma_{z,k}''(0)$ and
$Q_k(z) =\II\left(\frac{\partial}{\partial\theta_k}\Big|_z,
\frac{\partial}{\partial\theta_k}\Big|_z\right)$ as defined above
have the same values, but computing the geodesic curvatures of
$\gamma_{z,k}$ involve applying a renormalized metric and the
corresponding values are rescaled. In the next result we write down
the geodesic curvatures of the curves $\gamma_{z,k}$ for the Bergman
metric and describe its asymptotic behavior. We also express such
curvatures in terms of the second fundamental form $\II$. These
facts will allow us to compare our present situation with a more
general asymptotic behavior discussed in the next section.

\begin{theorem}\label{th:ball-bergman-curv}
Let $h_{\B^n}$ be the Riemannian metric associated to the Bergman
metric of $\B^n$ given as above and denote with $\|\cdot\|_{\B^n}$
the norm that it defines on tangent vectors. Then, for every
$z\in\widehat{\B}^n$ the geodesic curvature $\widehat{C}_k(z)$ of
the curve $\gamma_{z,k}$ at $z$ for the metric $h_{\B^n}$ satisfies the relations:
\begin{enumerate}
\item $\widehat{C}_k(z) =
    -\left\|\frac{\partial}{\partial\theta_k}\Big|_z\right\|_{\B^n}^{-2}
     \left\|\frac{\partial}{\partial r_k} \Big|_z\right\|_{\B^n}^{-1}
     h_{\B^n}\left(Q_k(z),
                \frac{\partial}{\partial r_k} \Big|_z\right),
$
\item $\widehat{C}_k(z) = \frac{2}{\sqrt{n+1}}C_k(z)$, where $C_k$ is given as in Theorem \ref{th:ball-curv}.
\end{enumerate}
In particular, $\widehat{C}_k(z)$ is, up to a sign, the norm of the
orthogonal projection of the vector $\left\|\frac{\partial}{\partial
\theta_k}\Big|_z\right\|_{\B^n}^{-2} Q_k(z)$ onto
$\frac{\partial}{\partial r_k} \Big|_z$ with respect to $h_{\B^n}$.
And we also have:
$$
      \widehat{C}_k(z) \rightarrow \frac{2}{\sqrt{n+1}}, \quad
                      \mbox{ as } z \rightarrow u.
$$
for any $u \in \partial\B^n$ such that $u_k \neq 0$.
\end{theorem}
\begin{proof}
The first relation follows from the definition of the geodesic
curvature as above applied to the new metric $h_{\B^n}$, the fact
that:
$$
    Q_k(z) = \II\left(\frac{\partial}{\partial\theta_k}\Big|_z,\frac{\partial}{\partial\theta_k}\Big|_z\right)
$$
and the corresponding relation of (1) in Theorem \ref{th:ball-curv}
for $\widehat{C}_k$.

The second relation is a consequence of the fact that $h_{\B^n} =
\frac{n+1}{4}h$ for $h$ the Riemannian metric on $\B^n$ rescaled so
that its sectional curvature lies in $[-1,-1/4]$.
\end{proof}

We specify now the results obtained in Sections \ref{se:Reinh-Berg}
and \ref{se:Reinh-Toepl} for the unit ball $\mathbb{B}^n$. The base
$\tau(\mathbb{B}^n)$  of  $\mathbb{B}^n$ has obviously the form
\begin{equation*}
\tau(\mathbb{B}^n) = \{r =(r_1,...,r_n): r^2=r_1^2+...+r_n^2 \in [0,1) \}.
\end{equation*}
As a custom in operator theory (see, for example, \cite{Zhu05}), introduce the family of weights
\begin{equation*}
\mu_{\lambda}(|z|)= c_{\lambda}\, (1-|z|^2)^{\lambda},
\end{equation*}
where the normalizing constant
\begin{equation*}
c_{\lambda} = \frac{\Gamma(n+\lambda+1)}{\pi^n \Gamma(\lambda+1)}
\end{equation*}
is chosen so that $\mu_{\lambda}(|z|)dv(z)$ is a probability measure in $\mathbb{B}^n$.

Introduce $L_2(\mathbb{B}^n, \mu_{\lambda})$ and its Bergman
subspace $\mathcal{A}^2_{\lambda}(\mathbb{B}^n) =
\mathcal{A}^2_{\mu_{\lambda}}(\mathbb{B}^n)$. It is well known (see,
for example, \cite{Zhu05}), that the Bergman projection
$B_{\lambda}$ of $L_2(\mathbb{B}^n, \mu_{\lambda})$ onto
$\mathcal{A}^2_{\lambda}(\mathbb{B}^n)$ has the form
\begin{equation*}
(B_\lambda \varphi)(z) = \int_{\mathbb{B}^n} \varphi(\zeta)\, K_{\lambda}(z, \zeta)\, \mu_{\lambda}(|\zeta|)dv(\zeta),
\end{equation*}
where the (weighted) Bergman kernel is given by
\begin{equation*}
K_{\lambda}(z, \zeta) = \frac{1}{(1-\sum_{k=1}^n z_k \overline{\zeta}_k)^{n+1+\lambda}}.
\end{equation*}

To calculate the constant $\alpha_p$, $p=(p_1,...,p_n) \in
\mathbb{Z}^n_+$, see (\ref{eq:alpha_p}), consider the integral
\begin{eqnarray*}
\int_{\mathbb{B}^n} |z^p|^2 \mu_{\lambda}(|z|) dv(z) &=&
\int_{\mathbb{B}^n} |z_1|^{2p_1}\cdot ... \cdot |z_n|^{2p_n}
\mu_{\lambda}(r) dv(z) \\
&=& \int_{\mathbb{T}^n} \prod_{k=1}^n \frac{dt_k}{it_k}
\int_{\tau(\mathbb{B}^n)} r_1^{2p_1} \cdot ... \cdot
r_n^{2p_n} \mu_{\lambda}(r) \prod_{k=1}^n r_kdr_k \\
&=& (2\pi)^n \alpha_p^{-2}.
\end{eqnarray*}
From the other hand side, by \cite{Zhu05}, Lemma 1.11, we have
\begin{equation*}
\int_{\mathbb{B}^n} |z^p|^2 \mu_{\lambda}(|z|) dv(z) =
\frac{p!\, \Gamma(n + \lambda + 1)}{\Gamma(n + |p| + \lambda + 1)},
\end{equation*}
that is,
\begin{equation*}
\alpha_p = \left(\frac{(2\pi)^n \, \Gamma(n + |p| + \lambda + 1)}{p! \, \Gamma(n + \lambda + 1)}\right)^{1/2}.
\end{equation*}

Now Theorem \ref{th:RT_aR*} for the case of the unit ball reads as follows.

\begin{theorem} \label{th:RT_aR*_ball}
Let $a=a(r)$ be a bounded measurable separately radial function. Then the
Toeplitz operator $T_a$ acting on $\Aa^2_{\lambda}(\mathbb{B}^n)$ is
unitary equivalent to the multiplication operator $\gamma_aI=R\,T_a
R^*$ acting on $l_2(\Z_+^n)$, where $R$ and $R^*$ are given by
(\ref{eq:unitR}) and (\ref{eq:unitR*}) respectively. The sequence
$\gamma_{a,\lambda}=\{\gamma_{a,\lambda}(p)\}_{p \in \Z_+^n}$ is
given by
\begin{eqnarray*} \label{eq:gamma_ball}
 \gamma_{a,\lambda}(p)&=& \frac{2^n \, \Gamma(n + |p| + \lambda + 1)}{p! \, \Gamma(\lambda + 1)} \int_{\tau(\mathbb{B}^n)} a(r)\,r^{2p}\, (1-r^2)^{\lambda}\, \prod_{k=1}^n r_kdr_k \\
&=& \frac{\Gamma(n + |p| + \lambda + 1)}{p! \, \Gamma(\lambda + 1)} \int_{\Delta(\mathbb{B}^n)} a(\sqrt{r})\,r^p\, (1-(r_1+...+r_n))^{\lambda}\, dr, \ \ \ \ \ \ p \in \Z_+^n,
\end{eqnarray*}
where $\Delta(\mathbb{B}^n)=\{r=(r_1,...,r_n):\, r_1+...+r_n \in
[0,1), \ r_k \geq 0, \ k=1,...,n\}$, $dr=dr_1...dr_n$, and
$\sqrt{r}=(\sqrt{r_1},...,\sqrt{r_n})$.
\end{theorem}

\section{Asymptotic geometric behavior of the $\T^n$-orbits in Reinhardt domains}
As before, let $D$ be a bounded logarithmically convex complete
Reinhardt domain with Bergman metric $ds^2_D$, associated Riemannian
metric $h_D$ and with $\|\cdot\|_D$ denoting the norm defined by
$h_D$ on tangent vectors.

Also, we will continue denoting with $\II$ the second fundamental
form of the foliation $\Oa$ by $\T^n$-orbits in $\widehat{D}$. As in
the case of the unit ball, by Proposition \ref{prop:II-quadratic},
$\II$ is completely determined by the vector fields:
\begin{eqnarray*}
  Q_k &=& \II\left(\frac{\partial}{\partial\theta_k},\frac{\partial}{\partial\theta_k}\right) \\
  Q_{kl} &=& \II\left(\frac{\partial}{\partial\theta_k} +
\frac{\partial}{\partial\theta_l},\frac{\partial}{\partial\theta_k}
+ \frac{\partial}{\partial\theta_l}\right).
\end{eqnarray*}
The norm of such vector fields was computed in the previous section
for the unit ball and such norm was related to the geodesic
curvature of suitable circles contained in complex geodesics. In
this section we will study the asymptotic behavior towards the
boundary of similar values for a more general Reinhardt domain.

As in the case of the unit ball, on our given Reinhardt domain, we
will consider for every $z\in \widehat{D}$ and $k=1,\dots,n$ the
curve:
$$
  \gamma_{z,k}(s) = (z_1,\dots,z_{k-1},e^{is}z_k,z_{k+1},\dots,z_n).
$$
Then, the proofs of Lemmas \ref{le:geodesics-O} and
\ref{le:Q-values} apply to our current more general setup without
change to conclude that $\gamma_{z,k}$ is an integral curve of
$\frac{\partial}{\partial\theta_k}$ and that we can write:
$$
    Q_k(z) = \II(\gamma_{z,k}'(0),\gamma_{z,k}'(0)) =
    \gamma_{z,k}''(0).
$$
As in the case of the unit ball, to better understand the asymptotic
behavior of the values of $Q_k$ one considers its normalized value
obtained by dividing by $\|\gamma_{z,k}'(0)\|_D^2 =
\|\frac{\partial}{\partial\theta_k}|_z\|_D^2$, i.e.:
$$
    \left\|\frac{\partial}{\partial\theta_k}\Big|_z\right\|_D^{-2}Q_k(z) =
    \left\|\frac{\partial}{\partial\theta_k}\Big|_z\right\|_D^{-2}\II\left(\frac{\partial}{\partial\theta_k}\Big|_z,\frac{\partial}{\partial\theta_k}\Big|_z\right)
    = \frac{\gamma_{z,k}''(0)}{\|\gamma_{z,k}'(0)\|_D^2},
$$
which now depends only on the direction associated to
$\gamma_{z,k}'(0) = \frac{\partial}{\partial\theta_k}|_z$ and not on
its magnitude. Moreover, the above identities show that
$\left\|\frac{\partial}{\partial\theta_k}\Big|_z\right\|_D^{-2}Q_k(z)$
measures both the extrinsic curvature of the foliation $\Oa$, given
by $\II$, and the curvature of $\gamma_{z,k}$, given by its
acceleration. For the unit ball it was proved that such vector field
is collinear with $\frac{\partial}{\partial r_k}$, and so to measure
its magnitude in that case it was enough to consider the norm of its
orthogonal projection onto $\frac{\partial}{\partial r_k}$. In our
more general setup,
$\left\|\frac{\partial}{\partial\theta_k}\right\|_D^{-2}Q_k$ may not
be collinear with $\frac{\partial}{\partial r_k}$, but we can still
consider the properties of the orthogonal projection of the first
onto the latter.

The previous discussion suggests to define for every $z\in
\widehat{D}$ and $k=1,\dots,n$:
$$
    \widehat{C}_k(z) = - \left\|\frac{\partial}{\partial\theta_k}\Big|_z\right\|_D^{-2}
            \left\|\frac{\partial}{\partial r_k}\Big|_z\right\|_D^{-1}
            h_D\left(Q_k(z),
            \frac{\partial}{\partial r_k}\Big|_z\right),
$$
which thus provides a measure of both the extrinsic curvature of the
foliation $\Oa$ on $\widehat{D}$ and the curvature of
$\gamma_{z,k}$.

Note that for the unit ball endowed with the Bergman metric, Theorem
\ref{th:ball-bergman-curv} shows that $\widehat{C}_k(z)$ is
precisely the geodesic curvature of $\gamma_{z,k}$ in the complex
geodesic $\phi_{z,k}$ considered in the previous section. Moreover,
such Theorem \ref{th:ball-bergman-curv} describes the asymptotic
behavior of $\widehat{C}_k$ towards the boundary in the case of the
unit ball. The main goal of this section is to prove that such
asymptotic behavior remains valid for suitable domains. More
precisely, we have the following result. We recall that $D$ is said
to have $\delta$ as a defining function if $D = \{z\in\C^n :
\delta(z) < 0\}$.

\begin{theorem}\label{th:Reinhardt-asymptotic}
Let $D$ be a bounded strictly pseudoconvex complete Reinhardt domain
with smooth boundary and with a smooth defining function $\delta$.
Then:
$$
    \widehat{C}_k(z) \rightarrow \frac{2}{\sqrt{n+1}}, \quad
    \mbox{ as } z \rightarrow u,
$$
for any $u\in \partial D$ such that $u_k \neq 0$ and
$\frac{\partial\delta}{\partial r_k}(u) \neq 0$.
\end{theorem}

We observe that for the unit ball we can take $\delta(z) = -1 +
\sum_{j=1}^n |z_j|^2 = -1 + \sum_{j=1}^n r_j^2$, and so the
conditions $u_k \neq 0$ and $\frac{\partial\delta}{\partial r_k}(u)
\neq 0$ are equivalent in this case.

Note that Theorems \ref{th:Reinhardt-asymptotic} and
\ref{th:ball-bergman-curv} together show that, under suitable
convexity and smoothness conditions on the domain $D$, the extrinsic
geometry of the foliation $\Oa$ in $\widehat{D}$ has exactly the
same asymptotic behavior towards the boundary as the one found for
the unit ball, at least with respect to the values of $Q_k$.

To prove Theorem \ref{th:Reinhardt-asymptotic} we will use the
expression of the metric $h_D$ in terms of the Bergman kernel $K_D$
from Theorem \ref{th:metric-Reinhardt} and the following celebrated
result by C.~Fefferman that describes the Bergman kernel of strictly
pseudoconvex domains with smooth boundary. This result appears as a
Corollary in page 45 of \cite{Fefferman74}.

\begin{theorem}[C.~Fefferman]\label{th:Fefferman}
If $D$ is a strictly pseudoconvex domain with smooth boundary and
smooth defining function $\delta$, then there exists $\varphi, \psi
\in C^\infty(\overline{D})$ with $\varphi$ nonvanishing in $\partial
D$, such that:
$$
    K_D(z,z) = \varphi(z)(-\delta(z))^{-(n+1)} + \psi(z)\log(-\delta(z))
$$
for every $z\in D$.
\end{theorem}

We first express the value of $\widehat{C}_k$ in terms of the
Bergman kernel $K_D$.

\begin{lemma}\label{le:Ck-Reinhardt}
Let $D$ be a bounded logarithmically convex complete Reinhardt
domain with Bergman kernel $K_D$. Then:
$$
    \widehat{C}_k =
    \frac{2\left(\frac{\partial K_D}{\partial r_k}\right)^3 - 3K_D\frac{\partial K_D}{\partial r_k}\frac{\partial^2 K_D}{\partial r_k^2}
    + K_D^2\frac{\partial^3 K_D}{\partial r_k^3} - \frac{3K_D}{r_k}\left(\frac{\partial K_D}{\partial r_k}\right)^2
    + \frac{3K_D^2}{r_k}\frac{\partial^2 K_D}{\partial r_k^2} + \frac{K_D^2}{r_k^2}\frac{\partial K_D}{\partial r_k}}%
    {\left(-\left(\frac{\partial K_D}{\partial r_k}\right)^2 + K_D\frac{\partial^2 K_D}{\partial r_k^2}
    + \frac{K_D}{r_k}\frac{\partial K_D}{\partial r_k}\right)^{\frac{3}{2}}},
$$
where $K_D$ and its partial derivatives are computed for the
function $z \mapsto K_D(z,z)$.
\end{lemma}
\begin{proof}
Let us denote with $\Gamma_{\theta_k \theta_l}^{r_j},\dots$, the
Schwarz-Christoffel symbols for the Levi-Civita connection of the
Riemannian manifold $(D,h_D)$ and with $h_{\theta_k
\theta_l},\dots$, the coordinate functions of the metric. By
Corollary \ref{cor:framing} and the definition of $\II$ it follows
that:
$$
    Q_k(z) = \sum_{l=1}^n \Gamma_{\theta_k \theta_k}^{r_l}\frac{\partial}{\partial r_l}.
$$
By using the well known formula that expresses the
Schwarz-Christoffel symbols in terms of the functions $h_{\theta_k
\theta_l},\dots$, and its partial derivatives (see \cite{Oneill83})
we have:
$$
    \Gamma_{\theta_k \theta_k}^{r_l} = -\frac{1}{2}
    \sum_{j=1}^n h^{r_l r_j}\frac{\partial h_{\theta_k \theta_k}}{\partial r_j},
$$
where, as usual, $h^{r_l r_j}$ denotes the entries of the inverse of
the matrix $(h_{r_l r_j})_{lj}$. We have used here that, by Theorem
\ref{th:metric-Reinhardt}, the functions $h_{\theta_k r_l} = 0$.
Hence, it follows that:
\begin{eqnarray*}
  \widehat{C}_k &=& - \left\|\frac{\partial}{\partial\theta_k}\Big|_z\right\|_D^{-2}
            \left\|\frac{\partial}{\partial r_k}\Big|_z\right\|_D^{-1}
            h_D\left(Q_k(z),
            \frac{\partial}{\partial r_k}\Big|_z\right) \\
   &=& \frac{1}{2}h_{\theta_k \theta_k}^{-1} h_{r_k r_k}^{-1/2}
   \sum_{l,j=1}^n h_{r_k r_l}h^{r_l r_j}\frac{\partial h_{\theta_k \theta_k}}{\partial r_j} \\
   &=& \frac{1}{2}h_{\theta_k \theta_k}^{-1} h_{r_k r_k}^{-1/2}
   \frac{\partial h_{\theta_k \theta_k}}{\partial r_k}.
\end{eqnarray*}
Again by Theorem \ref{th:metric-Reinhardt} we have $h_{r_k r_k} =
F_{kk}$ and $h_{\theta_k \theta_k} = r_k^2F_{kk}$, where:
\begin{equation} \label{eq:Fkk}
    F_{kk}(z) = \frac{1}{4}
        \left(\frac{\partial^2}{\partial r_k^2}
        + \frac{1}{r_k}\frac{\partial}{\partial r_k}
        \right) \log K_D(z,z),
\end{equation}
from which we obtain:
\begin{equation}\label{eq:Ck}
    \widehat{C}_k = \frac{1}{2F_{kk}^{3/2}}
        \left(\frac{\partial F_{kk}}{\partial r_k} + \frac{2}{r_k}F_{kk}\right).
\end{equation}
Then the result follows by computing $F_{kk}$ and $\frac{\partial
F_{kk}}{\partial r_k}$ in terms of $K_D$ with the use of equation
(\ref{eq:Fkk}) and replacing into equation (\ref{eq:Ck}).
\end{proof}

The following result can be proved easily using induction.

\begin{lemma}\label{le:K-partials}
Let $D$ be a strictly pseudoconvex domain with smooth boundary and
smooth defining function $\delta$. Let $\varphi, \psi \in
C^\infty(\overline{D})$ be the smooth functions from Theorem
\ref{th:Fefferman}. Then, for every $k=1,\dots,n$ and $j \geq 0$ we
have:
$$
    \frac{\partial^j K_D}{\partial r_k^j} =
        \sum_{l=0}^j \varphi_{jl} (-\delta)^{-(n+1+l)} +
        \sum_{l=1}^j \psi_{jl} \delta^{-l} +
        \psi_{j0}\log(-\delta),
$$
where the partial derivatives are computed for the function $z
\mapsto K_D(z,z)$, and the functions $\varphi_{jl}, \psi_{jl}$ are
given inductively by the following conditions:
\begin{enumerate}
  \item $\varphi_{00} = \varphi$, $\psi_{00} = \psi$,
  \item $\varphi_{jl} = \psi_{jl} = 0$ if either $j$ or $l$ is
    negative,
  \item for $j \geq 1$:
    \begin{eqnarray*}
    \varphi_{jj} &=& (n+j)\varphi_{j-1,j-1}\frac{\partial\delta}{\partial r_k} \\
    \psi_{jj} &=& -(j-1)\psi_{j-1,j-1}\frac{\partial\delta}{\partial r_k},
    \end{eqnarray*}
  \item for $0\leq l < j$:
    \begin{eqnarray*}
    \varphi_{jl} &=&    \frac{\partial \varphi_{j-1,l}}{\partial r_k} +
            (n+l)\varphi_{j-1,l-1}\frac{\partial\delta}{\partial r_k} \\
    \psi_{jl} &=&       \frac{\partial \psi_{j-1,l}}{\partial r_k}
            -(l-1)\psi_{j-1,l-1}\frac{\partial\delta}{\partial r_k},
    \end{eqnarray*}
\end{enumerate}
\end{lemma}

The next result provides an expression of $\widehat{C}_k$ in terms
of the defining function $\delta$. We observe that Theorem
\ref{th:Reinhardt-asymptotic} is now an easy consequence of such
expression.

\begin{theorem}
Let $D$ be a bounded strictly pseudoconvex complete Reinhardt domain
with smooth boundary and with a smooth defining function $\delta$.
If $\varphi$ is the function given by Theorem~\ref{th:Fefferman},
then:
$$
    \widehat{C}_k = \frac{a(-\delta)^{-(3n+6)} + b\delta^{-(3n+5)}}{(c(-\delta)^{-(2n+4)} + d\delta^{-(2n+3)})^{3/2}}
$$
on $\widehat{D}$, where $a,b,c,d \in C^\infty(\widehat{D})$ satisfy:
\begin{enumerate}
  \item $a = 2(n+1)\varphi^3\left(\frac{\partial \delta}{\partial
        r_k}\right)^3$,
  \item $c = (n+1)\varphi^2\left(\frac{\partial \delta}{\partial
        r_k}\right)^2$,
  \item $b,d$ extend continuously to $\{z\in\overline{D}: z_k \neq
  0\}$.
\end{enumerate}
\end{theorem}
\begin{proof}
We use Lemmas \ref{le:Ck-Reinhardt} and \ref{le:K-partials} to
express $\widehat{C}_k$ in terms of $\log(-\delta)$ and powers of
$\delta$ with smooth coefficients in $\widehat{D}$. Then the result
is simply a matter of identifying the coefficients in such
expression. The functions $a,c$ correspond to the lowest powers of
$\delta$ in the numerator and the denominator, respectively. The
functions $b,d$ involve terms of the form $\delta^l$ and $\delta^l
\log(-\delta)$ with $l\geq 1$ and powers of $1/r_k$, all of which
can be extended continuously to $\{z\in\overline{D}: z_k \neq 0\}$.
\end{proof}

\end{document}